\documentclass[11pt]{article}
\usepackage{amsmath,amssymb,amsthm}
\usepackage{multicol}
\usepackage{color}
\usepackage[latin1]{inputenc}
{\catcode `\@=11 \global\let\AddToReset=\@addtoreset}
\AddToReset{equation}{section}

\newtheorem{theorem}{Theorem}[section]

\newtheorem{proposition}{Proposition}[section]
\newenvironment{definition}{\begin{@definition}\rm}{\end{@definition}}
\newtheorem{@definition}{Definition}[section]
\newenvironment{remark}{\begin{@remark}\rm}{\end{@remark}}
\newtheorem{@remark}{Remark}[section]

\def\N{\mathbb{N}}
\def\Z{\mathbb{Z}}
\def\R{\mathbb{R}}
\def\C{\mathbb{C}}

\def\E{\mathrm E}
\def\P{\mathrm P}
\def\text{\mbox}
\def\1{{\bf 1}}

\def\i{\mathrm i}
\def\d{\mathrm d}
\def\e{\mathrm e}

\def\vep{\varepsilon}

\newcommand{\noi}{\noindent}
\newcommand {\nn}{\nonumber}

\begin{document}

\title{
Joint temporal and contemporaneous aggregation of\\ random-coefficient AR(1) processes with infinite variance
}

\author{Vytaut\.e  Pilipauskait\.e$^{1}$,  \ Viktor Skorniakov$^{2}$, \ Donatas Surgailis$^{2}$}
\date{\today \\ \small
	\vskip.2cm
	$^1$Aarhus University, Department of Mathematics, Ny Munkegade 118, 8000 Aarhus C, Denmark\\
	$^2$Vilnius University, Faculty of Mathematics and Informatics, Naugarduko 24, 03225 Vilnius, Lithuania}

\maketitle

\begin{abstract}
	
We discuss joint temporal and contemporaneous aggregation of $N$ independent copies of random-coefficient AR(1) process driven by i.i.d.\ innovations in the domain of normal attraction of an $\alpha$-stable distribution, $ 0< \alpha \le 2,$ as both $N$ and the time scale $n$ tend to infinity, possibly at a different rate.
Assuming that the tail distribution function of the random autoregressive coefficient regularly varies at the unit root with exponent $\beta > 0$,
we  show that, for $\beta < \max (\alpha, 1)$,
the joint aggregate displays a variety of stable and non-stable limit behaviors with stability index depending on $\alpha$, $\beta$ and the mutual increase rate of $N$ and $n$.
The paper extends the results of Pilipauskait\.e and Surgailis (2014) from $\alpha =2$ to $0 < \alpha < 2$.

\end{abstract}

\begin{quote}



\end{quote}

\section{Introduction}

Contemporaneous aggregation of random-coefficient AR(1) (RCAR(1)) processes is an important model for long-range dependence (LRD, also often referred to as long memory) in econometrics, see Granger \cite{gran80}, Robinson  \cite{robi78}, Zaffaroni \cite{zaff04}, Beran et al.\ \cite{ber13}. It explains how LRD can arise in a time series of macroeconomic variable, which is aggregate such as average or sum over a very large number of 
different micro-variables, each evolving by AR(1) with a random coefficient.
The concentration of the distribution of a random autoregressive coefficient $a$ at the unit root $a=1$, governed by the parameter $\beta > 0$ in
\begin{equation} \label{atail}
\P(a > 1-x) \sim \operatorname{const} 
x^\beta,  \quad x \downarrow  0,
\end{equation}
determines various properties of both the RCAR(1) process and the (limit) aggregate.
Particularly, for $1 < \beta < 2$, the RCAR(1) process 
exhibits LRD in the sense that its autocovariance function is absolutely non-summable since it decays slowly like $t^{1-\beta}$ as the time lag $t$ between two observations increases.
Furthermore, the limit of the normalized aggregate of independent RCAR(1) processes is a Gaussian process, which has the same autocovariance function and can obey a particular case of ARFIMA model, see Granger \cite{gran80} and the review paper Leipus et al.\ \cite{lei2014}.

Statistical inference for RCAR(1) model, especially estimation of the distribution of a random coefficient, has been extensively studied, see Robinson \cite{robi78}, Beran et al.\ \cite{ber2010}, Celov et al.\ \cite{celo07, celo2010}, Jirak \cite{jir13},
Leipus et al.\ \cite{lei2006, lei2016, lei2018a}.
Most of these papers deal with a panel $\{ X_i (t), \, i=1, \dots, N, \, t = 1, \dots, n\}$ of $N$ independent RCAR(1) processes observed over the time-period of length $n$. As $N$ and $n$ increase, possibly at different rate, statistical (dependence) properties of such a panel are determined by the parameter $\beta$ in \eqref{atail}.  Particularly, Pilipauskait{\. e} and Surgailis \cite{pili14} proved that, for $1 < \beta < 2$, the distribution of the sample mean $(Nn)^{-1} \sum_{i=1}^N \sum_{t=1}^n X_i (t)$ is asymptotically normal if $N^{1/\beta}/n \to \infty$, and it is symmetric $\beta$-stable if $N^{1/\beta}/n \to 0$. In the `intermediate' case $N^{1/\beta}/n \to \mu \in (0,\infty)$, this limit distribution is more complicated and has an integral representation with respect to (w.r.t.) a certain Poisson random measure. 
Leipus et al.\ \cite{lei2018b} studied the limit distribution of sample variance and sample covariances for such an RCAR(1) panel.

All the above works refer to the case of finite-variance innovations, however, the RCAR(1) model with infinite variance also presents
considerable interest since heavy tails are important in financial modeling (see, e.g.\ Mikosch \cite{miko03} and the references
therein). Puplinskait\.e and Surgailis \cite{pupl10}
studied contemporaneous aggregation of independent copies $\{ X_i (t), \, t \in \Z \}$, $i=1,2, \dots$, of an RCAR(1) process
\begin{equation} \label{AR}
X (t) = a X (t-1) + \vep (t), \quad t\in \Z,
\end{equation}
where $\{ \vep (t), \, t \in \Z \}$ is a sequence of i.i.d.\ random variables (r.v.s) belonging to the domain of normal attraction of an $\alpha$-stable distribution, $0 < \alpha \le 2$, and the autoregressive coefficient $a \in [0,1)$ is an r.v.\ independent of $\{ \vep (t), \, t \in \Z \}$  and having a density $\phi (x)$, $x \in [0,1)$, such that
\begin{equation}
\phi(x) \sim  \psi_1 (1-x)^{\beta-1}, \quad x \uparrow 1,
\label{betasing}
\end{equation}
for some $\beta > 0$ and $\psi_1 > 0$.
In \cite{pupl10} it was proved that, for $\beta > 1$, the normalized aggregate $\{ N^{-1/\alpha} \sum_{i=1}^N X_i(t), \allowbreak \, t \in \Z\}$ tends (in the sense of weak convergence of finite-dimensional distributions) to the $\alpha$-stable mixed moving average process $\bar X$ given by
\begin{eqnarray}\label{barX}
 \bar X(t) := \sum_{s\le t} \int_{[0,1)} x^{t-s} M_s(\d x), \quad t \in \Z,
\end{eqnarray}
where $\{ M_s (\d x), \, s \in \Z \}$ are independent copies of an $\alpha$-stable random measure $M(\d x)$ on $[0,1)$ with
control measure $\P(a \in \d x)$. For $1<\beta<\alpha$, the limit aggregate $\bar X$ has distributional LRD in the sense that its partial sums normalized by $n^{H}$, $H := 1- (\beta-1)/\alpha \in (1/\alpha, 1)$, tend to an $\alpha$-stable, $H$-self-similar process $\Lambda_{\alpha, \beta}$
with stationary dependent increments. See Section~\ref{sec2} for its definition. 

In this paper we study joint temporal and contemporaneous aggregation of independent copies of RCAR(1) process in \eqref{AR},
driven by i.i.d.\ $\alpha$-stable or related infinite variance innovations with a random autoregressive coefficient
as in \eqref{betasing}. We assume that both the number $N$ of individual processes and the time scale $n$ tend to infinity, possibly
at a different rate and extend the results of Pilipauskait\.e and Surgailis \cite{pili14}, who considered the finite variance case $\alpha =2$. It turns out that, similarly
to \cite{pili14}, the limit behavior of the joint aggregate
\begin{equation} \label{aggreSum}
S_{N,n}(\tau) :=  \sum_{i=1}^N \sum_{t=1}^{[n\tau]} X_i (t), \quad \tau \ge 0,
\end{equation}
depends on $\beta $ and the mutual increase rate of $N, n$;
moreover, it also depends on $\alpha$ leading to a complex  panorama of the limit distributions.
Theorem~\ref{thm2} below provides a nearly complete description of these limit distributions of suitably normalized $S_{N,n}= \{ S_{N,n} (\tau), \, \tau \ge 0\}$ in terms of parameters $0 < \alpha \le 2,$
$\beta >0 $  (with exception of $\alpha= \beta$  and $\alpha<\beta=1$), as $N, n \to \infty$.
In Table~\ref{table1} we summarize the results of Theorem \ref{thm2} for the sample mean
$(Nn)^{-1} S_{N,n} (1)$, including the cases when the mean of $X(t) $ and $S_{N,n}(1)$ does not exist.
\begin{table}[htbp] 	
	\begin{center}
		\begin{tabular}{ l  l  l }
			\hline
			{Parameter region} & {Mutual increase rate of $N, n$} & Limit distribution\\
			\hline
			$1 \le \beta < \alpha$ & {$N^{1/\beta}/n \to \infty $} & $\alpha$-stable
			\\
			$0  <\beta  < \min(\alpha,1)$ & {$N^{1/\beta}/n \to \infty $} & $(\alpha\beta)$-stable \\
			$0 < \beta < \alpha$ & {$N^{1/\beta}/n  \to 0$} & $\beta$-stable\\
			$0 < \beta < \alpha$ & $N^{1/\beta}/n  \to
			\mu \in (0,\infty) $  & `intermediate Poisson' \\
			\hline
			$\alpha < \beta < 1$ & {$N^{1/\gamma\beta}/n \to \infty $} & $(\alpha\beta)$-stable \\
			$\alpha < \beta < 1$ & $N^{1/\gamma\beta}/n \to 0 $ & $\alpha$-stable \\
			$\alpha < \beta < 1$ & $N^{1/\gamma\beta}/n \to \mu \in (0,\infty)$  & $(\alpha\beta)$-stable + $\alpha$-stable \\			
			\hline
			$ \beta > \max(\alpha,1)$ & arbitrary & $\alpha$-stable \\
			\hline
		\end{tabular}
	\end{center}
	\caption{Limit distribution of the sample mean or $S_{N,n} (1)$ in \eqref{aggreSum},
for $0 < \alpha \le 2$, $\beta > 0$ with $\gamma := \frac{1-\alpha}{1-\beta}$.
	}\label{table1}
\end{table}

The description in Table~\ref{table1}
is not very precise and needs some comments. Let us first note that the stable distributions in Table~\ref{table1} are generally
not symmetric and in some cases they are supported on $\R_+ := (0, \infty)$. The terminology `intermediate Poisson'
(borrowed from \cite{pili16,lei2018b}) refers to a certain infinitely divisible distribution written as an integral w.r.t.\ a Poisson random measure. The sum `$(\alpha\beta)$-stable + $\alpha$-stable' in Table~\ref{table1} indicates the convolution
of two distributions with different stability indices (a rather unusual 
result in limit theorems of the probability theory).

Intuitively, the results in Table~\ref{table1} 
can be explained by discussing the results of Theorem~\ref{thm1}  which deals with the iterated
limits of suitably normalized $S_{N,n}$ in \eqref{aggreSum} when first $N \to \infty $ and then $n \to \infty $,  or vice versa.
The iterated limits
are generally simpler to derive, and the joint limits in Theorem~\ref{thm2} can be regarded as some kind of `interpolation' between the
former limits. These limits are generally different in different parameter regions 
leading to three
parameter
regions: (i) $0<\beta <\alpha$, (ii) $0<\alpha<\beta <1$, and (iii) $\beta>\max(\alpha,1)$ in Theorem~\ref{thm2}.

First, let us note that, for $\beta > \max(\alpha, 1)$, our
all limits are relatively simple and coincide since
$S_{N,n}(1)$ behaves as a sum $\kappa^{1/\alpha}_\alpha \sum_{i=1}^N \sum_{t=1}^{n} \vep_i(t) $ of i.i.d.\ r.v.s in the domain of attraction of an $\alpha$-stable distribution with $\kappa_\alpha := \E(1-a)^{-\alpha} < \infty$; see the proof of Theorem~\ref{thm2}(iii). Hence,
we can turn our attention to  the parameter region $ 0< \beta < \max(\alpha, 1)$, where the iterated limits depend on the order and so the joint limits depend on the mutual rate of $N, n \to \infty$.  Let us note that in the region (i) $0<\beta<\alpha$ the results
of Theorems \ref{thm1} and \ref{thm2} naturally extend those of \cite{pili14} from $\alpha = 2$ to $0<\alpha < 2 $,
whereas in the parameter region (ii) $0< \alpha<\beta<1$ (which does not occur in \cite{pili14})
they are less predictable and somewhat surprising.

The iterated limits $\lim_{n \to \infty} \lim_{N \to \infty} $ (relations  \eqref{it:i}, \eqref{it:i0} of Theorem~\ref{thm1}) essentially follow from
\cite{pupl10} since they reduce to the $\alpha$-stable partial sums limit $\Lambda_{\alpha, \beta}$ 
of $\bar X$ in \eqref{barX} for $1< \beta < \alpha$, while for $\beta < 1$, the limit aggregate $\bar X$ is a random $(\alpha \beta)$-stable constant $V_{\alpha, \beta} $,
see \cite[Proposition~2.3]{pupl10} and the proof of \eqref{it:i0} of Theorem~\ref{thm1}.
(However, the case $1=\beta < \alpha$ is more delicate and requires a separate treatment, see the
proof of \eqref{it}.)
These observations may explain the two first lines in Table~\ref{table1}.
The third line in Table~\ref{table1}
may be explained by the iterated limit  $\lim_{N \to \infty} \lim_{n \to \infty} $ in \eqref{it:ii}, which in turn
relies on the (conditional) $\alpha$-stable partial sums limit as $n \to \infty$ in \eqref{S1conv} of the AR(1) process \eqref{AR} for {\it fixed} $a \in [0,1)$.  Unconditionally, the last limits
have $\beta$-tails and then \eqref{it:ii} turns out to be a sub-$\alpha$-stable process
with $\beta$-stable finite-dimensional distributions in agreement with the third line of Table~\ref{table1}.


Obviously, the iterated limits are not useful to explain the fourth line in Table~\ref{table1} which is part
of Theorem~\ref{thm2}(i) and one of the main results of this paper.
The intermediate (Poisson) process ${\cal Z}_{\alpha,\beta} = \{{\cal Z}_{\alpha, \beta}(\tau), \, \tau \ge 0 \} $ is defined
in \eqref{calZfd} and discussed in Section~\ref{sec3}. There,
we give its integral representation w.r.t.\ a Poisson random measure on the product space $\R_+ \times D(\R)$,
where $D(\R)$ is the Skorohod space of cadlag functions on $\R$,
and study its properties. We show that ${\cal Z}_{\alpha,\beta}$ plays
a role of a bridge between the limiting processes in the extreme cases $\mu = \infty$ and $\mu = 0$ of Theorem~\ref{thm2}(i),
because it is asymptotically locally and globally self-similar with these processes being its tangent processes; see Proposition~\ref{prop1}(v).

Finally, let us turn our attention to lines 5-7 of Table~\ref{table1} (parameter region $0< \alpha < \beta < 1$), which may be described as the
very strong dependence ($\beta < 1$) and even stronger variability ($\alpha < \beta$) in the RCAR(1) model \eqref{AR}.
This `regime' is a new  one since it could not happen in \cite{pili14} where $\alpha = 2$.
The results are part of Theorem~\ref{thm2}(ii).
We see from the iterated limits in \eqref{it:i0} and  \eqref{it:ii1} that the joint limit `chooses' between
two extreme behaviors: the $(\alpha \beta)$-stable random line $\{ V_{\alpha \beta}\, \tau, \, \tau \ge 0 \}$ with `infinite memory of increments',
and the $\alpha$-stable L\'evy process  $\{\kappa^{1/\alpha}_\alpha \zeta_\alpha (\tau), \, \tau \ge 0\}$ with `zero memory of increments'.
The `winner' of this `competition of limit behaviors' is determined by equating respective normalizations:
$ n N^{1/(\alpha \beta)} = (Nn)^{1/\alpha} $ leads to $N = n^{\gamma \beta} $ with $\gamma $ as in Table~\ref{table1},
which agrees with Table~\ref{table1} and Theorem~\ref{thm2}(ii).  Needless to say, the above argument is heuristic,
the proof of Theorem~\ref{thm2} is more involved and does not follow from Theorem~\ref{thm1}.



The proofs in the present paper (as well as \cite{pili14,lei2018b} and some other related work) clearly profit
from the detailed structure of the pre-limit AR(1) process, raising the question of their robustness in a more
general context. Remark \ref{rem1} discusses possible extensions to higher order RCAR models which seem feasible but
technically not easy.
We also note that our results 
can be put in a general framework of 
limit theorems for spatio-temporal
models with LRD. See the doctoral dissertation \cite{pili17}. In particular, they are related to the studies of 
the accumulated workload in network traffic under LRD,
as the time scale $n$ and the number $N$ of independent sources simultaneously increase, possibly at a different rate.
See \cite{taqq97, miko02, gaig03, gaig06, kaj08, domb11}. See also \cite{pili14} for a comparison between the joint temporal and contemporaneous aggregation of RCAR(1) processes and that of network traffic models with finite variance and the corresponding limit processes.
Interestingly, the intermediate limit of the accumulated workload process has also an integral though different
representation w.r.t.\ a certain Poisson random measure and can be regarded as a `bridge' between limit processes arising in the other two scaling regimes.
We note that joint aggregation of some network traffic models with infinite variance and LRD was studied in Levy and Taqqu \cite{levy00}, Pipiras et al.\ \cite{pipi04},
Kaj and Taqqu \cite{kaj08}.

{\it Notation.}  In what follows, $C$ stands for a positive constant whose precise value is unimportant and may change from line to line. We denote by $=_{\rm d}$, $\to_{\rm d}$ the equality in distribution and convergence in distribution, respectively. We also write $\to_{\rm fdd}$ and $\operatorname{(fdd)} \lim$  for the weak convergence and limit of finite-dimensional distributions. 

\section{Main results}\label{sec2}

\subsection{Assumptions}

\begin{definition}
Let $0 < \alpha \le 2$. Write $\vep \in {\cal D}(\alpha)$ if the distribution of an r.v.\ $\vep$ satisfies the following conditions:
\begin{itemize}
    \item for $\alpha=2$, $\E\vep^2<\infty$;
    \item  for $0 < \alpha < 2$, there exist some finite constants $c_1,c_2 \ge 0$ such that
    \begin{equation*}
        \lim_{x\to\infty}x^{\alpha}\P(\vep>x)=c_1,\quad \lim_{x\to-\infty}|x|^{\alpha}\P(\vep \le x)=c_2,\quad c_1+c_2>0;
    \end{equation*}
    \item in addition to the above, $\E \vep = 0$  for $1 < \alpha \le 2$,  and, for $\alpha =1$, the distribution of
    $\vep$ is symmetric.
\end{itemize}
\end{definition}
\begin{remark}
Assumption $\vep \in {\cal D}(\alpha)$ implies that $\vep$ belongs to the domain of normal attraction of an $\alpha$-stable distribution. That is, for a sequence $\{\vep(t), \, t = 1, 2, \dots \}$
of independent copies of $\vep$,
\begin{eqnarray} \label{convLevy}
n^{- 1/\alpha}\sum_{t=1}^{[n\tau]}\vep(t) &\to_{\rm fdd}&  \zeta_\alpha(\tau),
\end{eqnarray}
where $\zeta_\alpha = \{\zeta_\alpha(\tau), \, \tau \ge 0 \} $ is an $\alpha$-stable L\'evy process having characteristic function (see \cite[pages 574--581]{fell71})
\begin{align}
    &\E \e^{{\i}\theta \zeta_\alpha(\tau)} 
    =\e^{- \tau |\theta|^{\alpha}\omega(\theta)},\quad \theta\in \R,  \quad \text{with}   \label{e:charFOfalphaStable1}\\
    &\omega(\theta):=\begin{cases}
        \frac{\Gamma(2-\alpha)}{1-\alpha} ( (c_1+c_2)\cos (\frac{\alpha \pi}{2} )-{\i}(c_1-c_2) \operatorname{sign}(\theta) \sin (\frac{\alpha \pi}{2} ) ),  &\alpha\ne 1,2, \\
        (c_1+c_2)\frac{\pi}{2}, &\alpha=1,\\
        \frac{1}{2} \E \vep^2,    &\alpha=2. 
    \end{cases}\label{omega}
\end{align}
Furthermore, assumption $\vep \in {\cal D}(\alpha)$ implies $\E |\vep|^p < \infty$ for any $0 < p < \alpha$.
\end{remark}

In what follows, we assume that $\{\vep (t), \, t \in \Z \}$ in \eqref{AR} are independent copies of $\vep \in {\cal D}(\alpha)$ for some $0 < \alpha \le 2$.
Moreover, we assume that $a$ is an absolutely continuous r.v.\ having density $\phi$ which is supported on $[0,1)$ and admits the representation
\begin{equation}\label{atail2}
\phi(u) = \psi(u)(1-u)^{\beta-1},  \quad u\in [0,1),
\end{equation}
for some $\beta >0$ and some integrable function $\psi (u)$, $u \in [0,1)$, having finite limit $\lim\limits_{u\uparrow 1}\psi(u) =:\psi_1>0$.
The same assumption is made in \cite{pili14, pupl10} and other related works.
Then there exists a unique stationary
solution of \eqref{AR} given by
\begin{equation}\label{AR2}
X(t) = \sum_{s \le t} a^{t-s} \vep (s), \quad t \in \Z,
\end{equation}
where the series on the r.h.s.\ of \eqref{AR2}
converges in $L^p$ for $0 < p < \alpha \min(\beta, 1)$ if $0 < \alpha < 2$;  and for $0 < p \le 2$ such that $p < 2\beta$ if $\alpha = 2$.
For almost every $a \in [0,1)$, the series on the r.h.s.\ of \eqref{AR2} converges conditionally a.s. and conditionally in $L^p$ for $0 < p < \alpha$ if $0 < \alpha < 2$;  and for $0 < p \le 2$ if $\alpha = 2$.
See \cite{pupl09} for details.

\subsection{Limiting processes}

For $1 < \beta < \alpha \le 2$, we define a stochastic process $\Lambda_{\alpha,\beta} = \{ \Lambda_{\alpha,\beta}
(\tau),\, \tau \ge 0 \}$  by
\begin{align}\label{Lambda}
\Lambda_{\alpha,\beta} (\tau) &:= \int_{\R_+ \times \R} f_\tau (x,s) M (\d x, \d s), \quad \text{where}\\
f_\tau (x,s) &:= \int_0^\tau \e^{-x (t-s)} \1 (s \le t) \d t, \quad \tau \ge 0, \ x > 0, \ s \in \R,\nonumber
\end{align}
and $M (\d x, \d s)$ is an $\alpha$-stable random measure on $\R_+ \times \R$ with a control measure
$m (\d x, \d s) := \psi_1 x^{\beta-1} \d x \d s $ such that
$\E \e^{ \i \theta M (B) } =  \e^{-|\theta|^\alpha \omega(\theta) m (B)}$, $\theta \in \R$, for every Borel set $B \subset \R_+ \times {\R}$ with $m (B) < \infty$ and $\omega$, $\psi_1 $ given in \eqref{omega}, \eqref{atail2}.
The process $\Lambda_{\alpha,\beta} $ was introduced in  \cite{pupl10}. It
is $\alpha$-stable, $H$-self-similar with $H = 1 - (\beta-1)/\alpha \in (1/\alpha, 1)$,
has stationary dependent increments,
and is related to the integrated superposition of Ornstein-Uhlenbeck processes discussed in Barndorff-Nielsen  \cite{barn01}. See also \cite{grah18}.
The joint characteristic function of  $\Lambda_{\alpha,\beta}$ is given by
\begin{align} \label{Zriba}
\E \exp\Big\{{\i}\sum_{j=1}^d \theta_j \Lambda_{\alpha,\beta} (\tau_j) \Big\}
&=\exp \Big\{ - \int_{\R_+ \times \R}
\big|\sum_{j=1}^d \theta_j f_{\tau_j} (x,s) \big|^\alpha \omega\big(\sum_{j=1}^d \theta_j f_{\tau_j} (x,s) \big) m ({\d}x, {\d}s)  \Big\},
\end{align}
for $\theta_j \in \R$, $\tau_j \ge 0$, $j=1, \dots, d$, $d\in \N$.
For $\alpha = 2$,  $\Lambda_{2,\beta}$ is a Gaussian process with
mean zero and the autocovariance function
\begin{equation}\label{Lambdacov}
\E \Lambda_{2,\beta} (\tau_1) \Lambda_{2,\beta} (\tau_2) = \E \vep^2
\int_{\R_+ \times \R} f_{\tau_1} (x,s) f_{\tau_2} (x,s) m (\d x, \d s) =
\frac{\sigma^2_\beta}{2} (\tau_1^{2H} + \tau_2^{2H} - |\tau_1-\tau_2|^{2H}), \quad \tau_1, \tau_2 \ge 0.
\end{equation}
It follows that $\Lambda_{2,\beta}$ is a fractional Brownian motion with
Hurst index $H = (3-\beta)/2$ and variance $\E \Lambda^2_{2,\beta} (1) =:
\sigma_\beta^2 = \psi_1 \Gamma (\beta-1) \E \vep^2 /(2-\beta)(3-\beta)$.

Next, for $0< \lambda < 1$, $0 < \alpha \le 2$, $\beta >0$, let $W_{\lambda,\alpha, \beta} >0$ be
a $\lambda$-stable r.v.\ with
Laplace transform
\begin{align}\label{Wdef}
\E \e^{  - \theta W_{\lambda,\alpha, \beta} }&=\e^{ - \kappa_{\lambda,\alpha, \beta} \theta^{\lambda}},  \quad  \theta \ge 0,  \quad \text{where} \\
\kappa_{\lambda,\alpha, \beta}&:=\psi_1\int_0^\infty (1- \exp \{ - (\lambda \alpha/\beta)^{-1} x^{-\beta/ \lambda} \}) x^{\beta-1} \d x
= \frac{\psi_1 \Gamma(1-\lambda)}{(\lambda \alpha/\beta)^\lambda \beta} > 0.  \nn
\end{align}
It is well-known (see, e.g., \cite[Theorem~2.6.1]{zolo86}) that the Laplace transform in \eqref{Wdef} extends to all complex numbers $\theta \in \C $ with ${\rm Re} (\theta) \ge 0$.
Assume that $W_{\lambda,\alpha,\beta}$ is independent of the L\'evy process  $\zeta_\alpha $ in \eqref{convLevy}.
Define
\begin{align}\label{VWdef}
V_{\alpha,\beta}&:= W^{1/\alpha}_{\beta,\alpha,\beta} \, \zeta_\alpha (1), \qquad 0 < \beta < 1, \\
{\cal W}_{\alpha,\beta}(\tau)&:= W^{1/\alpha}_{\beta/\alpha,\alpha,\beta} \, \zeta_\alpha (\tau), \quad \tau \ge 0,  \quad 0< \beta < \alpha. \nn
\end{align}
Then, using \eqref{Wdef}, we obtain for any  $\theta \in \R$,
\begin{align}  \label{VWchf}
\E \e^{\i \theta V_{\alpha,\beta}}
&=\E \e^{- |\theta|^\alpha \omega(\theta) W_{\beta, \alpha, \beta}}
=  \exp \{ - \kappa_{\beta,\alpha, \beta} |\theta|^{\alpha \beta} (\omega(\theta))^\beta \},  \quad   0< \beta < 1, \\
\E \e^{\i \theta {\cal  W}_{\alpha,\beta}(\tau)}
&=\E \e^{-\tau |\theta|^\alpha \omega(\theta) W_{\beta/\alpha, \alpha, \beta}}
=  \exp \{ - \kappa_{\beta/\alpha,\alpha, \beta} \tau^{\beta/\alpha} |\theta|^{\beta} (\omega(\theta))^{\beta/\alpha} \},  \quad 0< \beta < \alpha, \nn
\end{align}
where
\begin{equation}\label{kappas}
\kappa_{\beta,\alpha, \beta} = \frac{\psi_1}{\alpha^\beta \beta}\Gamma (1- \beta), \qquad
\kappa_{\beta/\alpha, \alpha, \beta} = \frac{\psi_1}{\beta} \Gamma (1 - \frac{\beta}{\alpha}).
\end{equation}
From \eqref{VWchf}, it follows that
r.v.s  $V_{\alpha,\beta} $ and ${\cal W}_{\alpha,\beta}(\tau)$  are stable with respective stability indices
$\alpha \beta  < \alpha$ and $\beta < \alpha$. In a similar way, it follows that
${\cal W}_{\alpha,\beta} = \{ {\cal W}_{\alpha,\beta}(\tau), \, \tau \ge 0\}$ has $\beta$-stable finite dimensional distributions. Following
\cite[Section~3.8]{samo1994}, we call the stochastic processes in \eqref{VWdef}  {\it sub-stable}.
We note that
${\cal W}_{\alpha,\beta}$ enjoys the stationary increment and $H$-self-similarity with $H = 1/\alpha $ properties which
it inherits from the L\'evy process $\zeta_\alpha$. For $\beta =1 < \alpha \le 2$, introduce also an $\alpha$-stable r.v.\
$V_{\alpha,1} $ with a characteristic function
\begin{equation}\label{V1chf}
\E \e^{\i \theta V_{\alpha,1}} = \e^{-(\psi_1/\alpha)|\theta|^\alpha \omega(\theta)}, \quad \theta  \in \R.
\end{equation} 
Since $\lim_{\beta \uparrow 1} (1-\beta)\Gamma(1-\beta) =1,   $ it follows that $(1-\beta)^{1/(\alpha \beta)} V_{\alpha,\beta}  \to_{\rm d} V_{\alpha,1}$
as $\beta \uparrow 1$. The above discontinuity of the distribution of $V_{\alpha,\beta}$ at $\beta =1 $ can be explained by the additional
logarithmic normalization in \eqref{it} and \eqref{i00}. 

Finally, for $0 < \beta < \alpha \le 2$,
we define a random process ${\cal Z}_{\alpha, \beta} = \{ {\cal Z}_{\alpha,\beta} (\tau), \, \tau \ge 0 \}$ through its
joint characteristic function:
\begin{equation} \label{calZfd}
\E \exp \Big\{ \i \sum_{j=1}^d \theta_j {\cal Z}_{\alpha,\beta} (\tau_j)  \Big\}
= \exp \Big\{ \psi_1  \int_{\R_+}  \Big(\exp \Big\{- \int_{\R} \big| \sum_{j=1}^d \theta_j f_{\tau_j} (x,s) \big|^\alpha
\omega \big( \sum_{j=1}^d \theta_j f_{\tau_j} (x,s) \big) \d s \Big\} -1 \Big) x^{\beta-1} \d x  \Big\},
\end{equation}
where $\theta_j \in \R$, $\tau_j \ge 0$, $j=1, \dots, d$, $d \in \N$ and $f_\tau (x,s)$ is given in \eqref{Lambda}. 
A stochastic integral representation and various properties of ${\cal Z}_{\alpha,\beta}$ are discussed in Section~\ref{sec3}.

\subsection{Limit theorems}


In Theorems~\ref{thm1} and \ref{thm2}, the process $S_{N,n} = \{ S_{N,n} (\tau), \, \tau \ge 0 \}$
is the joint aggregate in \eqref{aggreSum} of independent copies of the RCAR(1) process $X = \{ X(t), \, t \in \Z \}$ in \eqref{AR2} satisfying the above-stated assumptions for some $0 < \alpha \le 2$, some $\beta > 0$ and some $\psi_1 > 0$.
Theorem~\ref{thm1} discusses iterated limits when $N \to \infty $ followed by $n \to \infty $
(limits \eqref{it:i}, \eqref{it:i0}),
 or vice versa (limits \eqref{it:ii}, \eqref{it:ii1}). Let $\kappa_\alpha := \E (1-a)^{-\alpha} $ when
the last expectation exists.

\begin{theorem} \label{thm1}
Let $0 < \beta <  \max(\alpha,1)$. Then 
\begin{align}\label{it:i}
\operatorname{(fdd)} \lim_{n \to \infty}  \lim_{N \to \infty} n^{-1+(\beta-1)/\alpha} N^{-1/\alpha}  S_{N,n}(\tau)
&= \Lambda_{\alpha,\beta}(\tau), \quad 
1 < \beta < \alpha, \\
\label{it:i0}
\operatorname{(fdd)} \lim_{n \to \infty}  \lim_{N \to \infty} n^{-1} N^{-1/(\alpha \beta)}  S_{N,n}(\tau)
&= V_{\alpha,\beta} \,\tau, \hskip.8cm 
0 < \beta < 1, \\ \label{it}
\operatorname{(fdd)}\lim_{n\to \infty} \lim_{N\to\infty} n^{-1} (N \log N)^{-1/\alpha} S_{N,n} (\tau)
&=  V_{\alpha,1} \tau,
\hskip.8cm 1 = \beta < \alpha, 
\end{align}
and
\begin{align}
\label{it:ii}
\operatorname{(fdd)} \lim_{N \to \infty}  \lim_{n \to \infty} N^{-1/\beta} n^{-1/\alpha} S_{N,n}(\tau)
&= {\cal W}_{\alpha,\beta}(\tau), \qquad \ 
0 < \beta < \alpha,\\
\label{it:ii1}
\operatorname{(fdd)} \lim_{N \to \infty}  \lim_{n \to \infty} N^{-1/\alpha} n^{-1/\alpha} S_{N,n}(\tau)
&= \kappa^{1/\alpha}_\alpha \zeta_\alpha (\tau), \qquad 
 0 < \alpha < \beta < 1.
\end{align}

\end{theorem}

\medskip

The following Theorem~\ref{thm2} discusses joint limits of appropriately normalized $S_{N,n}$ under simultaneous increase of $N, n$.
As noted in the Introduction, these limits depend on the mutual  increase rate of $N, n$ and
the parameters $\alpha, \beta  $. In \eqref{iii1} below,  $V_{\alpha,\beta}$ and  $\zeta_\alpha $ are mutually independent.

\medskip

\begin{theorem} \label{thm2}
	
\noi (i)  Let $0 < \beta < \alpha$. Let $N,n \to \infty$ so that
\begin{equation}
\frac{N^{1/\beta}}{n} \to \mu \in [0,\infty].
\end{equation}
Then:
\begin{eqnarray}
	N^{-1/\alpha} n^{-1 + (\beta-1)/\alpha} S_{N,n}(\tau)&\to_{\rm fdd}&
\Lambda_{\alpha,\beta} (\tau), \hskip1.5cm \mu = \infty, \  1 < \beta < \alpha, \label{i} \\
	N^{-1/(\alpha \beta)} n^{-1} S_{N,n}(\tau)&\to_{\rm fdd}&V_{\alpha,\beta} \, \tau \hskip2cm
\mu = \infty, \ 0 < \beta < \min(\alpha,1), \label{i0} \\
(N \log (N/n))^{-1/\alpha} n^{-1} S_{N,n} (\tau) & \to_{\rm fdd}&  V_{\alpha,1} \tau, \hskip2cm
\mu = \infty, \  1 = \beta < \alpha,  \label{i00}   \\
N^{-1/\beta} n^{-1/\alpha} S_{N,n}(\tau)&\to_{\rm fdd}&{\cal W}_{\alpha,\beta}(\tau), \hskip1.5cm \mu = 0, \  0 < \beta < \alpha, \label{ii} \\
	N^{-1/\beta} n^{-1/\alpha} S_{N,n}(\tau)&\to_{\rm fdd}&\mu^{1/\alpha}
	{\cal Z}_{\alpha,\beta}(\tau/\mu), \hskip.5cm \mu \in (0,\infty), \ 0 < \beta < \alpha. \label{iii}
	\end{eqnarray}

\noi (ii) Let $0 < \alpha < \beta < 1$.  Let $N,n \to \infty$ so that
\begin{equation}
\frac{ N^{1/(\gamma\beta)}}{n} \to \mu \in [0,\infty], \quad \text{where } \gamma := \frac{1-\alpha}{1-\beta} > 1.
 \end{equation}
Then:
	\begin{eqnarray}
	N^{-1/(\alpha \beta)} n^{-1} S_{N,n} (\tau) &\to_{\rm fdd}&V_{\alpha,\beta} \, \tau,  \hskip3.9cm   \mu = \infty,
	\label{i1} \\
	(Nn)^{-1/\alpha} S_{N,n} (\tau) &\to_{\rm fdd}&\kappa^{1/\alpha}_\alpha \zeta_\alpha (\tau),  \hskip3.2cm \mu = 0,\label{ii1} \\
	(Nn)^{-1/\alpha} S_{N,n} (\tau) &\to_{\rm fdd}
&\mu^{(1/\alpha)-1} V_{\alpha,\beta} \,\tau + \kappa^{1/\alpha}_\alpha \zeta_\alpha(\tau), \quad  \mu \in (0,\infty).	\label{iii1}
	\end{eqnarray}

\noi (iii)  Let $\beta > \max(\alpha,1)$. Then, as $N, n \to \infty$ in arbitrary way,
	\begin{eqnarray}\label{lim:levy}
	(Nn)^{-1/\alpha} S_{N,n} (\tau) &\to_{\rm fdd}& \kappa^{1/\alpha}_\alpha \zeta_\alpha (\tau).
	\end{eqnarray}

\end{theorem}

\begin{remark} \label{rem1} We expect that results in Theorems \ref{thm1} and \ref{thm2}, as well as in \cite{pili14}, can be extended
to higher order RCAR models, making suitable assumptions about the mixing distribution. Particularly, Oppenheim and Viano \cite{oppe04} discussed
long memory properties of RCAR($2p$), $p\ge 1$, model with autoregressive polynomial having one positive, one negative and $p-1$ pairs of nonreal (complex conjugate) roots whose moduli are assumed to be independent 
r.v.s whose densities have power-law behavior at $1$,  
 similar to \eqref{betasing}, for possibly  different exponents $\beta_{i}$. As shown in \cite{oppe04},  these assumptions lead to oscillating asymptotics thus 
seasonal behavior of the autocovariance function of the RCAR process. A more forthright higher-order
version of the RCAR(1) equation in \eqref{AR} is the RCAR($p$) model with real positive roots, viz.,
\begin{equation} \label{ARp}
(1- a_1 B) \cdots (1-a_p B) X(t) = \vep(t), \quad t \in \Z,
\end{equation}
where $a_i \in [0,1)$, $i=1, \dots, p$, are independent r.v.s 
and $B X(t) = X(t-1)$ is the backward shift. The stationary solution
of \eqref{ARp} is written as  a MA process $X(t) = \sum_{s\le t} b(t-s) \vep(s)$, $t\in\Z$, where $b(t) := \sum_{0 \le s_1 \le \dots \le s_{p-1} \le t}
a_1^{s_1} a_2^{s_2-s_1} \cdots a_p^{t-s_{p-1}} $ satisfy
\begin{equation} \label{Bp}
\sum_{t=0}^\infty b(t) = \prod_{i=1}^p (1-a_i)^{-1} =: A.
\end{equation}
Particularly, 
it follows that given $a_1, \dots, a_p$, conditionally
\begin{equation} \label{AB}
n^{-1/2} \sum_{t=1}^{[nt]} X(t) \ \to_{\rm fdd} \ A B(\tau)
\end{equation}
which agrees with \eqref{S1conv} below for $p=1$, $\alpha =2$. In the case when each $a_i$ has a density satisfying a similar relation as in
\eqref{betasing} for some $\beta_i>0$, $\psi_{1i}>0$, $i=1,\dots,p$, 
the (random) factor $A$ in \eqref{AB} has a heavy-tailed distribution 
with tail parameter $\beta_{\min} := \min_{1\le i \le p} \beta_i$, see \cite[Corollary p.~245]{emb80},  
and we can expect that, for $\beta_{\min} < \alpha =2$, 
the suitably normalized iterated limit
$\operatorname{(fdd)} \lim_{N \to \infty}  \lim_{n \to \infty} 
S_{N,n}(\tau)$ is the sub-Gaussian process ${\cal W}_{2,{\beta_{\min}}}(\tau)$,  
following the proofs of Theorem~\ref{thm1}~\eqref{it:ii} or
\cite[Theorem~2.1~(2.10)]{pili14}. We also then expect that  the suitably normalized iterated limit $\operatorname{(fdd)} \lim_{n \to \infty}  \lim_{N \to \infty} S_{N,n}(\tau)$ is  a fractional Brownian motion with Hurst parameter 
$H = (3 - \beta_{\min})/2$ 
(or a stable self-similar process  
in the case when $X(t)$ has infinite variance). A challenging open problem is to make the
above argument rigorous and to extend it to joint limits of $S_{N,n}$ as in Theorem \ref{thm2}.

\end{remark}

\section{The  intermediate process}
\label{sec3}

This section discusses properties of the intermediate process ${\cal Z}_{\alpha, \beta}$ introduced in \eqref{calZfd}
via its finite-dimensional characteristic function. We study Poisson stochastic integral representation,
local and global self-similarity, a.s.\ continuity and other properties of  ${\cal Z}_{\alpha, \beta}$. The results extend
\cite[Proposition 3.1]{pili14} from $\alpha =2$ to $0< \alpha < 2$. Roughly speaking,
the Poisson integral  representation of  ${\cal Z}_{\alpha, \beta}$ is obtained by replacing
the Brownian motion in \cite{pili14} by L\'evy process  $\zeta_\alpha$. However, some properties of
${\cal Z}_{\alpha, \beta}$ are not `continuous' at $\alpha = 2$, particularly, the second moment of ${\cal Z}_{\alpha, \beta}$ does
not exist for $\alpha < 2$  while  ${\cal Z}_{2, \beta}$ may have higher moments than 2, see \cite{pili14}.
Clearly, these moment differences between the cases $\alpha < 2$ and $\alpha =2$ are related to the differences between
the  $\alpha$-stable  L\'evy process $\zeta_\alpha$, $\alpha < 2$, and the Brownian motion $\zeta_2 =  B$.

Assume that the homogeneous L\'evy process $\zeta_\alpha $ in \eqref{convLevy} is extended to the whole real line $\R $ and induces
a probability measure $\P_\alpha $ on the Borel sets of the Skorohod space $D(\R)$ of cadlag functions from $\R $ to $\R$.
We start with a family
\begin{equation}\label{zdef2}
z (\tau;x) := \int_{\R} f_\tau (x,s) \d \zeta_\alpha (s), \quad \tau \ge 0, \ x > 0,
\end{equation}
of integrated Ornstein-Uhlenbeck processes driven by $\zeta_\alpha$,
where $f_\tau (x,s)$ is defined in \eqref{Lambda}. 
The process ${\cal Z}_{\alpha, \beta}$ can be defined by `mixing' the above elementary processes of \eqref{zdef2} on the path space $D(\R)$ of the L\'evy process
as follows.

Let $N (\d x, \d \zeta)$ denote a Poisson random measure on the product space $\R_+ \times D(\R)$ with a
mean
$\nu (\d x, \d \zeta) = \psi_1 x^{\beta-1} \d x \times \P_\alpha (\d \zeta )$,
where $\psi_1 > 0$, $0 < \beta < \alpha \le 2$. 
Then ${\cal Z}_{\alpha, \beta} = \{ {\cal Z}_{\alpha, \beta} (\tau), \, \tau \ge 0 \}$ can be defined as a stochastic integral with
respect to the above Poisson measure:
\begin{equation}\label{calZ}
{\cal Z}_{\alpha,\beta} (\tau) := \int_{(0,1) \times D(\R)} z (\tau; x ) N(\d x, \d \zeta) +
\int_{[1,\infty) \times D(\R)} z (\tau; x ) \big( N(\d x, \d \zeta) - \nu (\d x, \d \zeta) \1 (\alpha > 1) \big).
\end{equation}
If $1 < \alpha \le 2$, $1/\alpha < \beta < \alpha$, then the two integrals in \eqref{calZ} can be combined into a single one:
\begin{equation}\label{calZ1}
{\cal Z}_{\alpha,\beta} (\tau) = \int_{\R_+ \times D(\R)} z (\tau; x) \big( N (\d x, \d \zeta) - \nu (\d x, \d \zeta) \big).
\end{equation}
These and other properties of ${\cal Z}_{\alpha,\beta}$ are stated in the following proposition (we  refer to \cite{pili14, rajp89} for general properties of stochastic integrals w.r.t.\ Poisson random measure).

\begin{proposition}\label{prop1}
\noi (i) The process ${\cal Z}_{\alpha,\beta}$ in \eqref{calZ} is well-defined for any $0 < \beta < \alpha\le 2$. It has stationary increments, infinitely divisible finite-dimensional distributions, and the joint characteristic function given by \eqref{calZfd}.

\medskip

\noi (ii) If $0 < \beta <\alpha <2$, then $\E |{\cal Z}_{\alpha,\beta} (\tau)|^p < \infty$ for any $0 < p < \alpha \min(\beta,1)$. If $0< \beta < \alpha = 2$, then $\E |{\cal Z}_{\alpha,\beta}(\tau)|^p < \infty$ for any $0 < p < 2 \beta$.

\medskip

\noi (iii) For $1 < \alpha \le 2$,  $1/\alpha < \beta < \alpha$, ${\cal Z}_{\alpha,\beta}$ can be defined as in \eqref{calZ1} and
$\E {\cal Z}_{\alpha,\beta} (\tau) = 0$. Moreover, $\E | {\cal Z}_{\alpha,\beta} (\tau) |^2 < \infty$ if and only if $1 < \beta < \alpha = 2$, in
which case
$$
\E [{\cal Z}_{2,\beta} (\tau_1) {\cal Z}_{2,\beta} (\tau_2) ]
= \frac{\sigma_\beta^2}{2} ( \tau_1^{2H} + \tau_2^{2H} - |\tau_1 -\tau_2|^{2H}), \quad \tau_1, \tau_2 \ge 0,
$$
where $H = (3-\beta)/2$ and $\sigma_\beta^2$ are the same as in \eqref{Lambdacov}.

\medskip

\noi (iv) For $1< \alpha \le 2$, $1/\alpha < \beta < \alpha$, ${\cal Z}_{\alpha,\beta}$ is a.s.\ continuous.

\medskip

\noi (v) (Asymptotic self-similarity.) As $c \to 0$,
\begin{eqnarray*}
c^{-1+(\beta-1)/\alpha} {\cal Z}_{\alpha,\beta}(c \tau) &\to_{\rm fdd}& \Lambda_{\alpha,\beta} (\tau), \qquad 1 < \beta < \alpha,\\
c^{-1} {\cal Z}_{\alpha,\beta} (c \tau) &\to_{\rm fdd}& V_{\alpha,\beta} \, \tau,  \hskip1.1cm 0 < \beta < \min(\alpha,1),\\
c^{-1} (\log(1/c))^{-1/\alpha} {\cal Z}_{\alpha,\beta} (c \tau) &\to_{\rm fdd}& V_{\alpha,1} \tau, \qquad \quad  1=\beta<\alpha,
\end{eqnarray*}
where $V_{\alpha,1}$, $V_{\alpha,\beta}$ and $\Lambda_{\alpha,\beta}$ are defined in \eqref{V1chf}, \eqref{VWdef} and \eqref{Lambda}, respectively.
For $0 < \beta < \alpha$, as $c \to \infty$,
\begin{eqnarray*}
c^{-1/\alpha} {\cal Z}_{\alpha,\beta} (c \tau) &\to_{\rm fdd}& {\cal W}_{\alpha,\beta} (\tau),
\end{eqnarray*}
where ${\cal W}_{\alpha,\beta}$ is defined in \eqref{VWdef}.
\end{proposition}

\begin{remark} With Proposition~\ref{prop1}(v) in mind, we may say
${\cal Z}_{\alpha, \beta}$ plays the role of a bridge between the limit processes
in Theorem~\ref{thm2}(i).
For $\alpha \neq 1$, the limit processes ${\cal W}_{\alpha,\beta}$, $\Lambda_{\alpha,\beta}$ and r.v.\ $V_{\alpha,\beta}$ in Proposition~\ref{prop1}(v)
have different stability indices $\beta$, $\alpha$ and $\alpha \beta$, respectively, so we conclude that one-dimensional distributions
${\cal Z}_{\alpha, \beta}(\tau)$ are  not stable. For $\alpha \neq 1$, the process ${\cal Z}_{\alpha, \beta} $ is also not self-similar,
because ${\cal W}_{\alpha,\beta}$, $\Lambda_{\alpha,\beta}$ and $\{ V_{\alpha,\beta} \,\tau, \, \tau \ge 0 \}$ have different self-similarity indices.
\end{remark}

\begin{remark}
	If ${\cal Z}_1, {\cal Z}_2, \dots, {\cal Z}_N$ are independent copies of ${\cal Z} := {\cal Z}_{\alpha,\beta}$, then, for any $N \in \N$,
	\begin{equation}\label{interSS}
	{\cal Z} (\tau/N^{1/\beta}) =_{\rm fdd} N^{-1/(\alpha\beta)-1/\beta}  \sum_{i=1}^N {\cal Z}_{i}(\tau).
	\end{equation}
Relation \eqref{interSS} follows from infinite divisibility of Poisson random measure $N(\d x, \d \zeta)$ in
the stochastic integral representation 	\eqref{calZ} or the characteristic function \eqref{calZfd}.
See also (\cite{pili14}, (3.30)) where the above property is related to  the {\it aggregate-similarity} property
introduced in Kaj \cite{kaj05}.
For $0<\beta<\min(\alpha,1)$,  \eqref{interSS} and Proposition \ref{prop1}(v) imply that
$N^{1/\beta} {\cal Z} (\tau/N^{1/\beta}) =_{\rm d} N^{-1/(\alpha\beta)} \sum_{i=1}^N {\cal Z}_i (\tau) \allowbreak \to_{\rm d} V_{\alpha,\beta} \tau$ as $N\to\infty$. It follows that for a fixed $\tau >0$, the (marginal) distribution of
${\cal Z} (\tau) \equiv {\cal Z}_{\alpha,\beta}(\tau)$ belongs to the domain of normal attraction of an $(\alpha\beta)$-stable distribution, that is,
${\cal Z}_{\alpha,\beta}(\tau) \in {\cal D}(\alpha\beta)$ except possibly for the case $\alpha\beta = 1$, when the distribution of
${\cal Z}_{\alpha,\beta}(\tau)$ is not symmetric. Similarly,
$N^{H/\beta} {\cal Z}(\tau/N^{1/\beta}) =_{\rm d} N^{-1/\alpha} \sum_{i=1}^N {\cal Z}_i (\tau) \to_{\rm d} \Lambda_{\alpha,\beta} (\tau)$, where $H=1-(\beta-1)/\alpha$, implying ${\cal Z}_{\alpha,\beta} (\tau) \in {\cal D}(\alpha)$ for $1<\beta<\alpha$. These facts
entail the precise asymptotic behavior of tail probabilities of ${\cal Z}_{\alpha,\beta} (\tau)$ for $\alpha < 2 $ and $\tau >0$ fixed,
particularly, they show that condition $p < \alpha \min(\beta,1)$ in
Proposition \ref{prop1} (ii) cannot be improved.

\end{remark}

\begin{remark}
   Let $0 < \alpha < \beta < 1$. The limit process $\{ {\cal Z}^*_{\alpha,\beta} (\tau) := V_{\alpha,\beta}\, \tau
   + \kappa^{1/\alpha}_\alpha \zeta_\alpha (\tau), \, \tau \ge 0 \}$ in \eqref{iii1},
   Theorem~\ref{thm2}(ii) can be also regarded  as a `bridge' between the other two limit processes in \eqref{i1} and \eqref{ii1} since it is
   both locally and globally asymptotically self-similar:
   \begin{eqnarray*}
   c^{-1} {\cal Z}^*_{\alpha,\beta} (c \tau) &\to_{\rm fdd}& V_{\alpha,\beta}\, \tau,  \hskip1.1cm \text{as } c \to 0,\\
    c^{-1/\alpha} {\cal Z}^*_{\alpha,\beta}(c \tau) &\to_{\rm fdd}& \kappa^{1/\alpha}_\alpha \zeta_\alpha (\tau),  \quad \text{as } c \to \infty.
   \end{eqnarray*}
\end{remark}

\section {Proofs}\label{sec4}

We first present some preliminary facts that will be used in the proofs.

Let $0 < \alpha \le 2$. The characteristic function
of a r.v.\  $\vep \in {\cal D}(\alpha)$ has the following  representation in a neighborhood
of the origin (see, e.g., \cite[Theorem~2.6.5]{ibra71}): there exists
an $\epsilon>0$ such that 
\begin{equation}
\E {\e}^{{\i} \theta \vep} = {\e}^{-|\theta|^\alpha \omega(\theta) h(\theta)} \qquad \text{for any}  \ \theta \in \R, \ |\theta| < \epsilon,
\label{chfvep}
\end{equation}
where $h (\theta)$ is a positive function
tending to 1  as $\theta \to 0$ and $\omega(\theta)= \omega ( \operatorname{sign}(\theta) )$ is the same as in \eqref{omega}.

For $a\in [0,1)$ and $n\in \N$, let $c_n (a,s) := \sum_{t=1}^n a^{t-s} \1 (s \le t)$ and note the
following elementary inequalities:
\begin{equation}\label{ineq:dct}
\sum_{s \le 0} |c_n(a,s)|^\alpha
\le \frac{1}{\min(\alpha,1) (1-a)} \min \big(n, \frac{1}{1-a} \big)^\alpha,
\quad \quad \sum_{s=1}^n |c_n (a,s)|^\alpha \le n \min \big(n, \frac{1}{1-a} \big)^\alpha.
\end{equation}


For $z \in \mathbb{C}$, $\operatorname{Re}(z) \le 0$, we have
\begin{equation}\label{ineq:exp}
| \e^{z} - 1 | \le  \min(2, |z|), \quad |\e^{z} - 1 - z| \le |z|^2.
\end{equation}


\bigskip

\begin{proof}[Proof of Theorem \ref{thm1}]
The iterated limits \eqref{it:i}, \eqref{it:i0} follow from \cite[Theorems~2.1, 3.1, Proposition~2.3]{pupl10}.

\medskip

\noi  {\it Proof of \eqref{it}.}  It suffices to prove that
\begin{equation}\label{lim}
\bar X_N (t) := (N\log N)^{-1/\alpha} \sum_{i=1}^N X_i (t) \to_{\rm fdd} V_{\alpha,1},
\end{equation}
or that, for any $d \in \N$ and $\theta_t \in \R$, $t=1,\dots,d$,
\begin{equation}\label{lim:chf}
\E \e^{ \i \sum_{t=1}^d \theta_t \bar X_N (t)} \to \E \e^{\i \sum_{t=1}^d \theta_t V_{\alpha,1}} = \e^\Theta \quad \text{with } \Theta := - \frac{\psi_1}{\alpha} \big| \sum_{t=1}^d \theta_t \big|^\alpha \omega \big( \sum_{t=1}^d \theta_t \big).
\end{equation}
 Since $X_1,\dots,X_N$ are independent copies of $X$ in \eqref{AR2}, the l.h.s.\ of \eqref{lim:chf} can be rewritten as
 $( 1+\frac{\Theta_{N}}{N} )^N$ 
with
\begin{align}
\Theta_{N}
&:= N \E \big[ \exp \big\{ \i (N \log N)^{-1/\alpha} \sum_{t=1}^d \theta_t X(t) \big\} - 1 \big]\nn \\
&= N \int_{[0,1)} \big( \prod_{s \in \Z} \E \exp \big\{ \i (N \log N)^{-1/\alpha} c (u,s) \vep (s) \big\} - 1 \big) \phi(u) \d u \nn \\
&= N \int_{[0,1)} \big( \exp\big\{ - \frac{K_N(u)}{ (1-u) N \log N }
 \big\}  - 1 \big) \phi(u) \d u \label{THN}
\end{align}
where
$$
c(u,s) := \sum_{t=1}^d \theta_t u^{t-s} \1 (s \le t), \qquad
K_N(u) := (1-u) \sum_{s \in \Z} \left| c(u,s) \right|^\alpha \omega \big( c (u,s) \big) h \big( (N\log N)^{-1/\alpha}c (u,s) \big),
$$
(in \eqref{THN} we used \eqref{chfvep} and the fact that $c(u,s)$ is bounded uniformly in $u \in [0,1), s \in \Z$).
For $\delta \in (0,1) $ split
\begin{align*}
\Theta_N
&= N \big\{\int_{0}^{1-\delta}   + \int_{1-\frac{\delta}{N}}^1  + \int_{1-\delta}^{1-\frac{\delta}{N}} \big\} \big( \exp\big\{ - \frac{K_N(u)}{ (1-u) N \log N } \big\}  - 1 \big)
\phi(u) \d u =:  \sum_{i=1}^3 \Theta_{N,\delta}^{i}.
\end{align*}
Then \eqref{lim:chf} or $\lim_{N \to \infty} \Theta_N = \Theta $ follows from
\begin{equation}\label{THNi}
\lim_{\delta \to 0} \limsup_{N \to \infty}|\Theta_{N,\delta}^{i}| =0, \quad i=1,2, \qquad \lim_{\delta \to 0} \limsup_{N \to \infty}|\Theta_{N,\delta}^3
- \Theta| =0.
 \end{equation}
Here, $| \Theta_{N,\delta}^2 |\le 2 N \int_{1-\frac{\delta}{N}}^1 \phi (u) \d u \le C \delta $ for all $N$ large enough, 
implying \eqref{THNi} for $i=2 $.
Next, using \eqref{ineq:dct},  $\sum_{s \in \Z} |c(u,s)|^\alpha
\le C (1-u)^{-1}$. Therefore, $|K_N(u)| \le C$, $u  \in [0,1)$, and,  by \eqref{ineq:exp}, we obtain
$| \Theta_{N, \delta}^1 | \le C (\log N)^{-1} \int_0^{1-\delta} (1-u)^{-1} \phi(u) \d u= C (\delta \log N)^{-1}$,
proving \eqref{THNi} for $i=1$.

Consider the last relation in \eqref{THNi}. 
In view of \eqref{atail2}, we can replace  $\Theta_{N,\delta}^3$ by
$\Theta_{N,\delta}^4:= \psi_1 N \int_{ 1-\delta}^{1-\frac{\delta}{N}} ( \exp \{ - \frac{K_N(u)}{(1-u)N} \allowbreak \frac{1}{\log N} \} - 1 ) \d u = \psi_1 \int_{\delta}^{\delta N} ( \exp \{ - \frac{K_N ( 1-\frac{x}{N})}{x \log N} \}  - 1 ) \d x$,
which in turn can be replaced by
\begin{align*}
\Theta_{N,\delta}^5 &:=
- \psi_1 \int_\delta^{\delta N}  \frac{K_N (1-\frac{x}{N})}{x \log N} \d x
\end{align*}
since $|\Theta_{N,\delta}^4 - {\Theta}^{5}_{N,\delta} | \le C \int_\delta^{\delta N} \frac{\d x}{(x \log N)^{2}} = o(1)$, $N \to \infty$, follows from
\eqref{ineq:exp}.
We can  rewrite  $\Theta $ in \eqref{lim:chf}  in a similar way:
\begin{align*}
\Theta = - \psi_1 K \int_{\delta}^{\delta N} \frac{\d x}{x \log N}  \quad \text{with } K := \frac{1}{\alpha} \big| \sum_{t=1}^d \theta_t \big|^\alpha \omega \big( \sum_{t=1}^d \theta_t \big).
\end{align*}
Thus, the last relation in  \eqref{THNi} follows from $\lim_{\delta \to 0} \limsup_{N \to \infty} \int_\delta^{\delta N} |K_N (1-\frac{x}{N})- K | \frac{\d x}{x \log N}
= 0$ or
\begin{eqnarray}\label{THN3}
\limsup_{N \to \infty} \sup_{\delta < x < \delta N} \big|K_N \big(1-\frac{x}{N}\big)- K\big| \le \epsilon(\delta),
 \end{eqnarray}
where $\lim_{\delta \to 0} \epsilon(\delta) = 0$.
To prove \eqref{THN3}, denote $|z|_\omega^\alpha := |z|^\alpha \omega (z)$, $z \in \R,$ and note that $\sup_{0 < x < \delta N} |K_N (1-\frac{x}{N}) -
\tilde K(\frac{N}{x})| = o(1)$, $N \to \infty$, where
\begin{align}
\tilde K(y)&:=\frac{1}{y} \sum_{s\in \Z} \big|\sum_{t=1}^d \theta_t \big( 1 - \frac{1}{y} \big)^{t-s} \1(s\le t)\big|^\alpha_\omega \nn \\
&=\int_{\R}  \big|\sum_{t=1}^d \theta_t \big(1 - \frac{1}{y} \big)^{t-[sy]} \1([sy]\le t)\big|^\alpha_\omega \d s\nn \\
&\to\int_{-\infty}^0 \big|\sum_{t=1}^d \theta_t  \e^{s} \big|^\alpha_\omega \d s \  = \ K, \quad y \to  \infty, \label{tildeK}
\end{align}
by the dominated convergence theorem (DCT) using
$1 - z \le \e^{-z}$, $z \ge 0$. Hence,
$\sup_{0 < x < \delta N} |\tilde K(\frac{N}{x}) - K| \le \epsilon (\delta) = o(1)$, $\delta \to 0$, implying
\eqref{THN3} and \eqref{THNi}. This completes the proof of \eqref{it}.


\medskip

\noi {\it Proof of \eqref{it:ii}}. Let us first prove that
\begin{equation}\label{S1conv}
n^{-1/\alpha} S_{n} (\tau) :=  n^{-1/\alpha} \sum_{t=1}^{[n\tau]} X (t) \to_{\rm fdd} (1-a)^{-1} \zeta_\alpha (\tau),
\end{equation}
where $\zeta_\alpha$, $X$ are the same as in \eqref{convLevy}, \eqref{AR2}, and $a \in [0,1)$ is fixed.
It suffices to show that, for any $d \in \N$ and $\tau_j > 0$, $\theta_j \in \R$, $j = 1, \dots, d$,
\begin{equation}\label{C1ch}
\E_a \exp \big\{ \i  n^{-1/\alpha} \sum_{j=1}^d \theta_j S_{n} (\tau_j) \big\} \to
\E_a \exp \big\{ \i (1-a)^{-1} \sum_{j=1}^d \theta_j \zeta_\alpha (\tau_j) \big\},
\end{equation}
where $\E_a [\cdot] = \E [\cdot |a ]$ stands for conditional expectation.
For brevity of notation, we restrict the proof of \eqref{C1ch} (as well as all the rest in this theorem) to $d = 1$ and $\tau_1 = \tau > 0$, $\theta_1 = \theta \in \R$.
Split
$
\E_a [\e^{\i \theta n^{-1/\alpha} S_n (\tau)} ( \1 (a \in I_n) + \1 (a \in I_n^c)  ) ] =: \Phi'_n (\theta,a) + \Phi''_n (\theta,a),
$
where $I_n := [0,1 - \frac{\log n}{n^{1/\alpha}} )$, $I_n^c := [0,1) \setminus I_n $. For $c_n(a,s)$ in \eqref{ineq:dct},
$\sup_{a \in I_{n}, s \in \Z} | n^{-1/\alpha}  c_{[n\tau]}(a,s) | = O( (\log n)^{-1} ) = o(1)$. Hence for all $n$ large enough,
we can use \eqref{chfvep} to rewrite  $\Phi'_n(\theta,a)$ as
$\Phi'_n (\theta,a) =  \e^{ - |\theta|^\alpha \omega (\theta)  K_n (a) } \1 (a \in I_n)$,
where
$$
K_n (a) := n^{-1} \sum_{s \in \Z} |c_{[n\tau]}(a,s) |^\alpha h ( \theta n^{-1/\alpha} c_{[n\tau]}(a,s)  )
$$
and
$
\sup_{a \in I_n, s \in \Z} | h ( \theta n^{-1/\alpha} c_{[n\tau]}(a,s) ) - 1 | = o (1).
$
Relation \eqref{C1ch} follows from
$\lim_{n \to \infty} K_n (a) =
\tau (1- a)^{-\alpha}$
for every $a \in [0,1)$ and $\lim_{n \to \infty} \Phi''_{n} (\theta,a) = 0$. Both these facts are completely
elementary, and we omit the details. This proves
\eqref{S1conv} for $d = 1$. The proof for $d > 1$ follows similarly.

Let $\{ (1-a_i)^{-1} \zeta_{\alpha, i} (\tau), \, \tau \ge 0 \}$, $i=1,2, \dots$, be independent copies of $\{ (1-a)^{-1} \zeta_\alpha (\tau), \, \tau \ge 0 \}$. With \eqref{S1conv} in mind, \eqref{it:ii} follows from
\begin{equation}\label{S2conv1}
N^{-1/\beta} \sum_{i=1}^N (1-a_i)^{-1} \zeta_{\alpha,i} (\tau) \to_{\rm fdd} {\cal W}_{\alpha,\beta} (\tau).
\end{equation}
Consider the one-dimensional convergence in \eqref{S2conv1} at $\tau = 1$. For $\theta \in \R$, we have
$
\E \exp \{ \i \theta  N^{-1/\beta}
\sum_{i=1}^N (1-a_i)^{-1} \zeta_{\alpha,i} (1)\} =
(\E \exp \{ \i \theta N^{-1/\beta} (1-a)^{-1} \zeta_\alpha (1) \} )^N = (1 + \frac{\Theta_N}{N} )^N,
$
where
$$
\Theta_N := N\E  [\e^{-N^{-\alpha/\beta} (1-a)^{-\alpha} |\theta|^\alpha  \omega (\theta)} - 1].
$$
We also have
$\E \e^{ \i \theta {\cal W}_{\alpha,\beta} (1)} = \e^{- \kappa_{\beta/\alpha,\alpha, \beta} |\theta|^\beta (\omega(\theta))^{\beta/\alpha}},
$
see \eqref{VWchf}. The desired convergence in \eqref{S2conv1} at $\tau = 1$ follows from
\begin{equation}\label{Uch1}
\Theta_N \to - \kappa_{\beta/\alpha,\alpha, \beta} |\theta|^\beta (\omega(\theta))^{\beta/\alpha}, \quad \forall \theta \in \R.
\end{equation}
By assumption \eqref{betasing}, there exists an $\epsilon > 0$ and a constant $C > 0$ such that $\phi (u) \le C (1-u)^{\beta-1}$ for all $u \in [1-\epsilon,1)$. Split
\begin{align*}
\Theta_N = \Theta^0_N + \Theta^1_N := N\E\big[ (\e^{- N^{-\alpha/\beta} (1-a)^{-\alpha} |\theta|^\alpha  \omega (\theta)} - 1 ) \big( \1 (0\le a < 1-\epsilon) + \1 (1-\epsilon \le a < 1) \big) \big],
\end{align*}
where $|\Theta^0_N| \le  N^{1-\alpha/\beta} C \epsilon^{-\alpha} |\theta|^\alpha \P ( 0 \le a < 1- \epsilon ) = o(1)$ since $\beta < \alpha.$
Hence, it suffices to prove \eqref{Uch1} for $\Theta^1_N$ instead of $\Theta_N$.
By change of a variable,
$$
\Theta^1_N = N \int_{1-\epsilon}^1 (\e^{-N^{-\alpha/\beta} (1-u)^{-\alpha} |\theta|^\alpha \omega (\theta) } - 1 ) \phi (u) \d u =  \int_0^{\epsilon N^{1/\beta}} (\e^{-x^{-\alpha} |\theta|^{\alpha} \omega (\theta) } - 1 ) N^{1-1/\beta} \phi \big( 1 - \frac{x}{N^{1/\beta}} \big) \d x,
$$
where $N^{1-1/\beta} \phi ( 1 - \frac{x}{N^{1/\beta}} ) \to \psi_1 x^{\beta-1}$ for $x \in \R_+$ by assumption \eqref{betasing}.
Moreover, $N^{1-1/\beta} \phi ( 1 - \frac{x}{N^{1/\beta}} ) \le C x^{\beta-1}$ for $x \in (0, \epsilon N^{1/\beta}]$. Therefore, the DCT implies that
$$
\Theta_N^1 \to \psi_1 \int_0^\infty (\e^{-x^{-\alpha} |\theta|^{\alpha} \omega (\theta) } - 1 ) x^{\beta-1} \d x = - \kappa_{\beta/\alpha,\alpha, \beta} |\theta|^\beta (\omega (\theta))^{\beta/\alpha},
$$
see  \eqref{kappas}, \eqref{Wdef}.  This proves \eqref{Uch1} and \eqref{S2conv1}.

\medskip


\noi {\it Proof of \eqref{it:ii1}}. Note that \eqref{S1conv} also holds for $\beta > \alpha$. Let $\{ (1-a_i)^{-1} \zeta_{\alpha, i} (\tau), \, \tau \ge 0 \}$, $i=1,2, \dots$, be as
in  \eqref{S2conv1}.     
It suffices to prove that
\begin{equation}\label{S2conv}
N^{-1/\alpha} \sum_{i=1}^N (1-a_i)^{-1} \zeta_{\alpha,i} (\tau) \to_{\rm fdd} \kappa^{1/\alpha}_\alpha \zeta_\alpha (\tau).
\end{equation}
Consider the one-dimensional convergence in \eqref{S2conv} at $\tau = 1$. For any $\theta \in \R$, we have
$
\E \exp \{ \i \theta  N^{-1/\alpha}
\sum_{i=1}^N (1-a_i)^{-1} \zeta_{\alpha,i}(1) \} =  (1 + \frac{\Theta_N}{N} )^N,
$
where
$
\Theta_N := N\E  [\e^{-N^{-1} (1-a)^{-\alpha} |\theta|^\alpha  \omega (\theta)} - 1] \to - \kappa_\alpha
|\theta|^\alpha  \omega (\theta) = \log \E \e^{ \i \theta \kappa^{1/\alpha}_\alpha \zeta_\alpha (1)}
$
with $\kappa_\alpha =\E  (1-a)^{-\alpha} < \infty$
by the DCT. The general finite-dimensional convergence in  \eqref{S2conv} follows in a similar way.
This proves \eqref{it:ii1} and completes the proof of
Theorem~\ref{thm1}.
\end{proof}


\begin{proof}[Proof of Theorem \ref{thm2}]
In each case of Theorem~\ref{thm2}, we will prove that, for any $d \in \N$, $0< \tau_1 < \dots < \tau_d<\infty$, and $\theta_j \in \R$, $j=1,\dots,d$, as $N,n \to \infty$,
\begin{equation}\label{lim:fdd}
\E \e^{ \i A_{N,n}^{-1} \sum_{j=1}^d \theta_j S_{N,n} (\tau_j) } \to \E \e^{\i \sum_{j=1}^d \theta_j {\cal S}(\tau_j)},
\end{equation}
where $A_{N,n} \to \infty$ denotes a sequence of normalizing constants and $\{ {\cal S} (\tau), \, \tau \ge 0 \}$ denotes the limit process. Since $X_1, X_2, \dots$ are independent processes, we can rewrite the l.h.s.\ of \eqref{lim:fdd} as $(1 + \frac{\Theta_{N,n}}{N} )^N$ and  reduce the proof to showing that
\begin{equation}\label{lim}
\Theta_{N,n} := N \E \big[ \exp \big\{ \i A_{N,n}^{-1} \sum_{j=1}^d \theta_j S_{1,n} (\tau_j) \big\} - 1 \big]
\to \log \E \e^{ \i \sum_{j=1}^{d} \theta_j {\cal S} (\tau_j) } =: \Theta 
\end{equation}
as $N,n \to \infty$, in each case of Theorem~\ref{thm2}. Conditioning on $a$, we have
\begin{equation}\label{def:vartheta}
\Theta_{N,n} = N \int_{[0,1)} \big[ \prod_{s \in \Z} \varphi_{\vep} \big( A_{N,n}^{-1} \vartheta_n (u,s) \big) - 1 \big] \phi(u) \d u \quad \text{with } \vartheta_n (u,s) := \sum_{j=1}^d \theta_j \sum_{t=1}^{[n\tau_j]} u^{t-s} \1 (s \le t),
\end{equation}
 where $\varphi_\vep (\theta) := \E \e^{\i \theta \vep}$, $\theta \in \R$, is the characteristic function of $\vep \in {\cal D}(\alpha)$.
Next, we  need to split the interval $[0,1) = I_{N,n} \cup I_{N,n}^c$ with $I_{N,n} := [0,1-u_{N,n})$, where $u_{N,n} \to 0$ is chosen so that
$\sup_{u \in I_{N,n}, s \in \Z} |\vartheta_n(u,s)| = O (u_{N,n}^{-1})= o(A_{N,n})$ and
$|N \E [ (\exp \{ \i A_{N,n}^{-1} \sum_{j=1}^d \theta_j S_{1,n} (\tau_j) \} - 1) \1 (a_1 \in I_{N,n}^c  ) ] | \le C N \P ( a_1 \in I_{N,n}^c ) \le C N \int_{I_{N,n}^c} (1-u)^{\beta-1} \d u = O (N u_{N,n}^\beta) = o (1) $ is negligible. By doing so, we obtain that  $h ( A_{N,n}^{-1} \vartheta_n ( u,s )) \to  1 $ uniformly in $u \in I_{N,n}, s \in \Z$, and, taking into account \eqref{chfvep}, \eqref{lim}--\eqref{def:vartheta},
\begin{equation*}
\Theta_{N,n} \sim N \int_{I_{N,n}} \Big( \exp \Big\{ - A_{N,n}^{-\alpha}  \sum_{s \in \Z} | \vartheta_n (u,s) |^\alpha \omega ( \vartheta_n (u,s) ) \Big\} - 1 \Big) \phi (u) \d u.
\end{equation*}
In order to avoid this rather tedious step,
from now on, we will assume that $h (\theta)\equiv 1$. That is, $\vep =_{\rm d} \zeta_\alpha (1)$ has a stable distribution and we can take $u_{N,n}\equiv 0$.
Moreover, for simplicity of exposition, in  cases (i) and (ii) of Theorem~\ref{thm2},
we also assume that $\phi (u) \equiv \psi_1 (1-u)^{\beta-1}$ in \eqref{def:vartheta}. Similar
simplifications are also imposed in the proof of \cite[Theorem~2.2]{pili14}.  Finally, since $\omega (\theta)$ depends on the sign
of $\theta $ alone, we shall assume that $\omega (\theta) \equiv 1 $. After all, $\Theta_{N,n} $ of \eqref{def:vartheta} reduces to
\begin{equation}\label{ThetaNn}
\Theta_{N,n} = \psi_1  N \int_{[0,1)} \big( \exp \big\{ - A_{N,n}^{-\alpha}
\sum_{s \in \Z} |\vartheta_n (u,s) |^\alpha  \big\} - 1 \big) (1-u)^{\beta -1} \d u.
\end{equation}
The only exception is the `short memory' case (iii) of Theorem~\ref{thm2} where we use \eqref{ThetaNn} with $\phi(u)$
instead of $\psi_1 (1-u)^{\beta-1}$, $u \in [0,1)$.

We consider each case of Theorem~\ref{thm2} separately.

\medskip
	
\noi {\it Proof of Theorem~\ref{thm2}(i).} Let $N,n \to \infty$ so that $\mu_{N,n} := N^{1/\beta}/n \to \mu \in [0,\infty]$ and set
\begin{eqnarray}
&B_{N,n} := \begin{cases}n, \\
N, \\
N^{1/\beta}, \\
N^{1/\beta},
\end{cases} \quad
A_{N,n} := n^{1+1/\alpha} \begin{cases}
\mu_{N,n}^{\beta/\alpha}, &\text{if } \mu = \infty, \ 1 < \beta < \alpha,\\
(\mu_{N,n} \log \mu_{N,n})^{1/\alpha}, & \text{if } \mu = \infty, \ 1 = \beta < \alpha,\\
\mu_{N,n}^{1/\alpha}, &\text{if } \mu = \infty, \ 0 < \beta < \min(\alpha,1),\\
\mu_{N,n}, &\text{if } \mu \in [0, \infty), \ 0 < \beta < \alpha.
\end{cases}
\end{eqnarray}
($B_{N,n}$ and $A_{N,n}$ are simultaneously defined in the above four cases of $\mu, \alpha, \beta $.)
Note that $A_{N,n}$ agree with respective normalizations in Theorem~\ref{thm2}(i).
After the change of variable $B_{N,n} (1-u) = x$, \eqref{ThetaNn} can be rewritten as
$\Theta_{N,n} = (\psi_1 N/B^\beta_{N,n}) \int_0^{B_{N,n}} (\e^{ - K_{N,n}(x)} -1 ) x^{\beta -1} \d x $, where
\begin{align}
K_{N,n}(x) &= K_{N,n}^- (x) + K_{N,n}^+ (x)
:= A_{N,n}^{-\alpha} \sum_{s \le 0} \big| \vartheta_n \big( 1-\frac{x}{B_{N,n}}, s \big) \big|^\alpha
+ A_{N,n}^{-\alpha} \sum_{s > 0} \big| \vartheta_n \big( 1-\frac{x}{B_{N,n}}, s \big) \big|^\alpha.
\label{def:KNn}
\end{align}

\smallskip

\noi {\it Proof of \eqref{iii}} (case $\mu \in (0,\infty)$, $0 < \beta < \alpha$). The r.h.s.\ of \eqref{lim} can be written as the integral
$\Theta = 
\psi_1  \int_0^\infty ( \e^{-K_\mu (x)} - 1 ) x^{\beta-1} \d x,$ see \eqref{calZfd},
where
\begin{equation}\label{def:Kmu}
K_\mu (x) := \mu \int_{\R} \big| \sum_{j=1}^d \theta_j f_{\tau_j/\mu} (x,s) \big|^\alpha 
\d s
\end{equation}
and $f_\tau (x,s)$ as in \eqref{Lambda}. Using $B_{N,n} = N^{1/\beta}$ and
writing the sums over integers $s$, $t$ in \eqref{def:KNn}, \eqref{def:vartheta} as integrals,
after the change of variables $s \to N^{1/\beta}s$,  $t \to N^{1/\beta} t$,
the l.h.s.\ of \eqref{lim} becomes 
$\Theta_{N,n} = \psi_1 \int_0^{N^{1/\beta}} (\e^{ - K_{N,n} (x) } - 1) x^{\beta-1} \d x$
with
$K_{N,n} (x) 
= \mu_{N,n} \int_\R | \tilde \vartheta_{N,n} (x,s) |^\alpha 
\d s$,
where
\begin{align}
\tilde \vartheta_{N,n} (x,s)
&:=\frac{1}{N^{1/\beta}} \vartheta_n \big( 1 - \frac{x}{N^{1/\beta}}, [N^{1/\beta} s] \big)\nn \\
&= \sum_{j=1}^d \theta_j
\int_{\R}
\big( 1 - \frac{x}{N^{1/\beta}} \big)^{[N^{1/\beta} t] - [N^{1/\beta} s]} \1 ( \max([N^{1/\beta} s],1) \le [N^{1/\beta} t] \le [n \tau_j] ) \d t\nn \\
&\to \sum_{j=1}^d \theta_j f_{\tau_j/\mu} (x,s) \label{lim:tildevartheta}
\end{align}
for any $x \in \R_+$, $s \in \R$.
Using $1-z \le \e^{-z}$, $z \in [0,1]$, we get the dominating bound $| \tilde \vartheta_{N,n} (x,s) | \le C f_{2 \tau_d/\mu} (x,s)$, 
$x \in (0, N^{1/\beta}]$, $s \in \R$, implying $K_{N,n} (x) \to K_\mu (x)$, $x \in \R_+$,
 by the  DCT.  Using \eqref{ineq:exp},
 we can extend the last dominating bound to $|K_{N,n}(x)| \le C \int_\R |f_{2\tau_d/\mu}(x,s)|^\alpha \d s$,
$x \in (0, N^{1/\beta}]$, where the last integral is estimated in \eqref{ineq:f}. Then, another application of the DCT yields the convergence in 
\eqref{lim}, proving  \eqref{iii}.

\medskip

\noi {\it Proof of \eqref{i}} (case $\mu = \infty$, $1 < \beta < \alpha$). By \eqref{Zriba}, the r.h.s.\ of \eqref{lim} equals to
$ 
 - \psi_1 \int_0^\infty K_1 (x) x^{\beta-1} \d x$, where  $K_1 (x)$ is as in \eqref{def:Kmu} (with $\mu = 1$).
After a change of variable,
the l.h.s.\ of \eqref{lim} can be written as
$\Theta_{N,n} = \psi_1 \int_0^n \mu_{N,n}^\beta ( \e^{ - \mu_{N,n}^{-\beta} K_{N,n} (x) } - 1 ) x^{\beta-1} \d x, $
with $K_{N,n} (x) := \int_{\R} | \tilde \vartheta_{N,n} (x,s) |^\alpha 
\d s $ and
$\tilde \vartheta_{N,n} (x,s) := \frac{1}{n} \vartheta_n ( 1 - \frac{x}{n}, [n s] ) \to \sum_{j=1}^d \theta_j f_{\tau_j} (x,s)$
for any $x \in \R_+$, $s \in \R$ (to justify the last relationship, use \eqref{lim:tildevartheta} with $N^{1/\beta}$ replaced by $n$). Therefore, $K_{N,n}(x) \to K_1 (x), \
\mu_{N,n}^\beta ( \e^{ - \mu_{N,n}^{-\beta} K_{N,n} (x) } - 1 ) \to -  K_1 (x), \  x \in \R_+ $ and, finally, the convergence in \eqref{lim}
follows using the DCT similarly to the proof before. This proves \eqref{i}.

\medskip

\noi {\it Proof of \eqref{i0}} (case $\mu = \infty$, $0 < \beta < \min(\alpha,1)$).
The r.h.s.\ of \eqref{lim} can be written as
\begin{eqnarray}\label{chfV}
&\Theta = 
-\kappa_{\beta,\alpha,\beta} \big| \sum_{j=1}^d \theta_j \tau_j \big|^{\alpha \beta} 
= \psi_1 \int_0^\infty \big( \exp \big\{- \frac{1}{\alpha x} \big|\sum_{j=1}^d \theta_j \tau_j \big|^\alpha
\big\} - 1 \big) x^{\beta-1} \d x,
\end{eqnarray}
see \eqref{Wdef}--\eqref{kappas}.
On the other hand, for the l.h.s.\ of \eqref{lim}, after a change of variable, we get
$\Theta_{N,n} = \psi_1 \int_0^{N^{1/\beta}} ( \e^{-K_{N,n}(x)} - 1 ) x^{\beta-1} \d x,$
where $K_{N,n} (x) = K_{N,n}^- (x) + K_{N,n}^+ (x)$ is given in \eqref{def:KNn} with $B_{N,n} = N^{1/\beta}$. Rewrite
$
K_{N,n}^- (x) = \int_{\R} | \tilde \vartheta_{N,n} (x,s) |^\alpha 
\d s,
$
where
\begin{eqnarray*}
&\tilde \vartheta_{N,n} (x,s) 
:= \sum_{j=1}^d \theta_j \int_{\R} \big( 1-\frac{x}{N^{1/\beta}} \big)^{[nt]-[N^{1/\beta}s]} \1 ( 1 \le [nt] \le [n\tau_j], \, [N^{1/\beta} s] \le 0 ) \d t
\to \sum_{j=1}^d \theta_j \tau_j \e^{x s} \1 (s \le 0)
\end{eqnarray*}
in view of $n/N^{1/\beta} \to 0$,
for any $x > 0$, $s \in \R, s \ne 0$. We have the dominating bound
$|\tilde \vartheta_{N,n}(x,s)| \le C \e^{xs} \1 (s \le C)$, $x \in (0, N^{1/\beta}]$, $s \in \R$.
Whence, 
\begin{equation}\label{lim:Kminus}
K_{N,n}^- (x) \to \frac{1}{\alpha x} \big| \sum_{j=1}^d \theta_j \tau_j \big|^\alpha \quad \text{and} \quad |K_{N,n}^- (x)| \le \frac{C}{x}
\end{equation}
for $x \in (0,N^{1/\beta}]$. Relation \eqref{ineq:dct} implies
$$
|K_{N,n}^+ (x)| \le \frac{C n}{A_{N,n}^{\alpha}} \min \big( n, \frac{N^{1/\beta}}{x} \big)^{\alpha} \le \frac{C}{\mu_{N,n}} \min \big(1,  \frac{\mu_{N,n}}{x} \big)^{\min({\alpha},1)} = o(1)
$$
uniformly in $x > 0$ as $N,n\to\infty$. Moreover, we conclude that $|K_{N,n}(x)| \le C x^{-\min(\alpha,1)}$ for $x \in (1,N^{1/\beta}]$. Since $0 < \beta < \min(\alpha,1)$, in view of \eqref{ineq:exp}, the DCT  implies 
\eqref{lim}, proving   \eqref{i0}.

\medskip

\noi 
	{\it Proof of  \eqref{i00}} (case $\mu = \infty$, $1=\beta < \alpha$). Write the r.h.s.\ and l.h.s.\
	of \eqref{lim} respectively as
	$\Theta = - \frac{\psi_1}{\alpha} | \sum_{j=1}^d \theta_j \tau_j |^\alpha $ and
	\begin{align*}
	\Theta_{N,n} = N  \psi_1\int_{[0,1)} \big( \exp \big\{ - \frac{ K_{N,n}(u)}{(1-u) N \log \frac{N}{n}} \big\} - 1 \big)  \d u \quad  \text{with }
	K_{N,n}(u) := n^{-\alpha} (1-u) \sum_{s \in \Z} |\vartheta_n (u,s)|^\alpha.
	\end{align*}
	For given $\delta \in (0,1)$ split
	\begin{align*}
	\Theta_{N,n} &= N  \psi_1 \big\{\int_0^{1- \frac{\delta}{n}} + \int_{1- \frac{\delta}{N}}^1 + \int_{1-\frac{\delta}{n}}^{1- \frac{\delta}{N}}\big\}\big( \exp \big\{ - \frac{K_{N,n}(u)}{(1-u) N \log \frac{N}{n}} \big\} - 1 \big)  \d u
	=:\sum_{i=1}^3 \Theta_{N,n,\delta}^i. 
	\end{align*}
	By  \eqref{ineq:dct} and \eqref{ineq:exp}, for every $\delta>0$,
	\begin{align*}
	|\Theta_{N,n,\delta}^1| &\le \frac{C}{n^\alpha \log \frac{N}{n}} \int_{0}^{1-\frac{\delta}{n}} \Big( \frac{1}{1-u} \min \big(n, \frac{1}{1-u} \big)^\alpha+ n \min \big(n,\frac{1}{1-u}\big)^\alpha \Big) \d u\\
	&\le \frac{C}{n^\alpha \log \frac{N}{n}} \int_{\delta}^{n} \Big( \frac{1}{x} \min \big( n,\frac{n}{x} \big)^\alpha + \min \big(n,\frac{n}{x} \big)^\alpha \Big) \d x = o(1) 
	\end{align*}
	as $N/n \to \infty $. Also, $|\Theta_{N,n,\delta}^2| \le C\delta $ can be made arbitrary small by a suitable choice of $\delta >0$.
	Consider
	\begin{align*}
	\Theta_{N,n,\delta}^3 = \psi_1 \int_{\delta}^{\delta \frac{N}{n}} \big( \exp \big\{ - \frac{K_{N,n} (1- \frac{x}{N} )}{x \log \frac{N}{n}} \big\} - 1 \big) \d x \sim - \psi_1 \int_{\delta}^{\delta \frac{N}{n}} \frac{K_{N,n} (1- \frac{x}{N} )}{x \log \frac{N}{n}} \d x =: \Theta_{N,n,\delta}^4
	\end{align*}
	since $|\Theta_{N,n,\delta}^3 - \Theta_{N,n,\delta}^4| \le C \int_\delta^{\delta (N/n)} \frac{\d x}{(x \log (N/n))^2}
	= o(1)$, as $N/n \to \infty$, follows from \eqref{ineq:exp} and the bound $\sup_{0< x < \delta N/n} | K_{N,n}(1-\frac{x}{N})| \le C$, which is a consequence of \eqref{ineq:dct}. Rewrite $\Theta $ similarly to $\Theta_{N,n,\delta}^4$, viz.,
	\begin{align*}
	\Theta = - \psi_1 K \int_{\delta}^{\delta \frac{N}{n}} \frac{\d x}{x \log \frac{N}{n}} \quad \text{with } K := \frac{1}{\alpha} \big| \sum_{j=1}^d \theta_j \tau_j \big|^\alpha. 
	\end{align*}
	Thus, \eqref{i00} follows (c.f.\ \eqref{THN3}) from
	\begin{eqnarray}\label{THN4}
	\limsup_{N,n, N/n \to \infty} \sup_{\delta < x < \delta (N/n)} \big|K_{N,n} \big( 1-\frac{x}{N} \big)- K\big| \le \epsilon(\delta)
	\end{eqnarray}
	where $\lim_{\delta \to 0} \epsilon(\delta) = 0$. Towards this end, write $K_{N,n} (1- \frac{x}{N} ) = \tilde K( \frac{N}{x}, n) $
	similarly as in \eqref{THN3} above, where
	\begin{eqnarray}
	\tilde K(y,z)
	:= \int_{\R} \Big| \sum_{j=1}^d \theta_j \int_{\R} \big( 1- \frac{1}{y} \big)^{[z t]-[ys]} \1(1\vee [ys] \le [zt] \le [z\tau_j]) \d t \Big|^\alpha \d s,
	\quad y, z >0.
	\end{eqnarray}
	Then using the DCT similarly as in \eqref{tildeK} above, it follows that $\lim_{y,z,y/z \to \infty} \tilde K(y,z) = K $, implying that
	$\limsup_{y,z,y/z \to \infty} \sup_{y/z > 1/\delta} |\tilde K(y,z) - K| \le \epsilon (\delta)$ with $\epsilon(\delta) $ as in
	\eqref{THN4}, hence \eqref{THN4}, thus completing the proof of   \eqref{i00}.

\medskip

\noi {\it Proof of \eqref{ii}} (case $\mu = 0$, $0< \beta < \alpha$). The r.h.s.\ of \eqref{lim} can be written as
\begin{align}\label{chfcalW}
\Theta&= \ 
\log \E \e^{ \i W^{1/\alpha}_{\beta/\alpha, \alpha,\beta} \sum_{i=1}^d (\sum_{j=i}^d \theta_j ) ( \zeta_\alpha (\tau_i) - \zeta_\alpha (\tau_{i-1}) )  }\nn \\
&= \ \log \E \e^{ - W_{\beta/\alpha,\alpha,\beta}
\sum_{i=1}^d (\tau_i - \tau_{i-1})  |\sum_{j=i}^d \theta_j |^\alpha 
} = - \kappa_{\beta/\alpha,\alpha,\beta} \big( \sum_{i=1}^d (\tau_i - \tau_{i-1} ) \big| \sum_{j=i}^d \theta_j \big|^\alpha 
\big)^{\beta/\alpha}\nn \\
&= \ \psi_1  \int_0^\infty \big( \exp \big\{ - \frac{1}{x^\alpha} \sum_{i=1}^d (\tau_i-\tau_{i-1}) \big| \sum_{j=i}^d \theta_j \big|^\alpha
\big\} - 1 \big) x^{\beta-1} \d x
\end{align}
with $\tau_0 := 0$ and $\kappa_{\beta/\alpha,\alpha,\beta}$ as in \eqref{kappas}.
After a change of variable, we get
$
\Theta_{N,n} = \psi_1 \int_0^{N^{1/\beta}} ( \e^{-K_{N,n} (x)} - 1 ) x^{\beta-1} \d x,
$
where $K_{N,n} (x) = K_{N,n}^- (x)+K_{N,n}^+ (x)$ is given as in \eqref{def:KNn} with $B_{N,n} = N^{1/\beta}$. Rewrite
$K_{N,n}^+ (x) = \int_{\R} | \tilde \vartheta_{N,n} (x,s) |^\alpha 
\d s,$
where
\begin{eqnarray*}
&\tilde \vartheta_{N,n} (x,s) 
:= \frac{1}{x} \sum_{j=1}^d \theta_j \big(1- \big( 1-\frac{x}{N^{1/\beta}} \big)^{[n\tau_j]-[ns]+1} \big) \1 (0 < [ns] \le [n\tau_j] )
\to \frac{1}{x} \sum_{j=1}^d \theta_j \1 (0 < s \le \tau_j)
\end{eqnarray*}
for $x \in \R_+$, $s \in \R$, since $n/N^{1/\beta} \to \infty$. Therefore,
$ 
K_{N,n}^+ (x) \to \frac{1}{x^\alpha} \sum_{i=1}^d (\tau_i-\tau_{i-1}) | \sum_{j=i}^d \theta_j |^\alpha 
$ 
for $x \in \R_+$. From \eqref{ineq:dct}, it follows that $| K_{N,n}^- (x) | \le C A_{N,n}^{-\alpha} ( \frac{x}{N^{1/\beta}} )^{-(1+\alpha)} = C \mu_{N,n} x^{-(1+\alpha)} = o (1)$ for $x \in \R_+$.
Moreover, $|K_{N,n}(x)| \le C x^{-\alpha}, \ x \in (1,N^{1/\beta}]$.
Since $0 < \beta < \alpha$, the DCT and \eqref{ineq:exp}
imply $\Theta_{N,n} \to \Theta$, proving \eqref{ii},  thereby completing the proof of
Theorem~\ref{thm2}(i).	

\medskip

\noi {\it Proof of Theorem~\ref{thm2}(iii)}. We use the representation in \eqref{ThetaNn} with $A_{N,n} = (N n)^{1/\alpha} $ and
$\phi(u)$
instead of $\psi_1 (1-u)^{\beta-1}$, $u \in [0,1)$. We have that $ \Theta_{N,n} = \int_0^1 N (\e^{-N^{-1} K_{n}(u)} - 1) \phi(u) \d u$, where
$K_{n} (u) = K^-_n(u) +  K^+_n(u) := n^{-1} \sum_{s \le 0} | \vartheta_n (u,s) |^\alpha
+ n^{-1} \sum_{s > 0} | \vartheta_n (u,s) |^\alpha$ with $\vartheta_n (u,s)$ defined in \eqref{def:vartheta}.
Rewrite
$ K_{n}^+ (u) = \int_{\R} |\vartheta_n(u,[ns])|^\alpha
\d s, $
	where
	$\vartheta_n (u,[ns]) = (1-u)^{-1} \sum_{j=1}^d \theta_j (1-u^{[n\tau_j]-[n s] +1}) \1 (0 < [n s] \le [n \tau_j]) \to (1-u)^{-1} \sum_{j=1}^d \theta_j \1 (0 < s \le \tau_j)$ for $u \in [0,1)$, $s \in \R$. Hence,
	\begin{eqnarray}\label{lim:Kplus}
&K_{n}^+ (u) \to (1-u)^{-\alpha} \sum_{i=1}^d (\tau_i - \tau_{i-1}) \big| \sum_{j=i}^{d} \theta_j \big|^\alpha
	\end{eqnarray}
	for $u \in [0,1)$, where $\tau_0 := 0$. From \eqref{ineq:dct}, for $u \in [0,1)$, we further obtain that $K^-_n(u)\to 0$
	and that
\begin{eqnarray} \label{Kplusbdd}
&|K_{n}^- (u)| \le
	\frac{C}{n(1-u)} \min \big(n, \frac{1}{1-u} \big)^\alpha
\le  C \begin{cases} (1-u)^{-\alpha},  &\alpha \ge 1, \\
(1-u)^{-1}, &\alpha < 1.
\end{cases}
\end{eqnarray}
Since $|K_{n}^+(u)| \le C (1-u)^{-\alpha}$, $|K_{n} (u)| \le C(1-u)^{-\max(\alpha,1)} =: \bar K(u)$, where $\int_0^1 \bar K(u) \phi (u) \d u < \infty  $ due to the fact that
$\beta > \max (\alpha, 1)$.
Then, by the DCT,
we conclude that
\begin{align*}
\Theta_{N,n} &\to - \kappa_\alpha \sum_{i=1}^d (\tau_i - \tau_{i-1}) | \sum_{j=i}^{d} \theta_j |^\alpha 
= \log \E \e^{\i \kappa_\alpha^{1/\alpha} \sum_{j=1}^{d} \theta_j \zeta_\alpha (\tau_j) },
	\end{align*}
where (recall) $\kappa_\alpha  = \E (1-a)^{-\alpha} < \infty$.
	Theorem~\ref{thm2}(iii) is proved.

\medskip

\noi {\it Proof of Theorem~\ref{thm2}(ii)}. Recall that $0< \alpha < \beta < 1 $.
Set $\mu_{N,n} := N^{1/(\gamma\beta)}/n \to \mu \in [0,\infty]$,  $A_{N,n} := N^{1/(\alpha \beta)} n$ if $\mu = \infty$;
and $A_{N,n} := (Nn)^{1/\alpha}$ if $\mu \in [0, \infty)$.
Consider  separately  terms $S_{N,n}^\pm (\tau)$  in the 
decomposition 
$S_{N,n}(\tau) = S_{N,n}^- (\tau) + S_{N,n}^+ (\tau)$, where
$$
S_{N,n}^- (\tau) := \sum_{i=1}^N \big( \sum_{t=1}^{[n\tau]} a_i^t \big) X_i(0), \quad S_{N,n}^+ (\tau) := \sum_{i=1}^N \sum_{s>0} \big( \sum_{t=1}^{[n\tau]} a_i^{t-s} \1 (s \le t) \big) \vep_i (s).
$$
From the proofs of \eqref{i0}, \eqref{lim:levy} (in particular, \eqref{lim:Kminus}, \eqref{lim:Kplus}, \eqref{Kplusbdd}) we 
see that
\begin{align}
N^{-1/(\alpha \beta)} n^{-1} S^-_{N,n} (\tau) &\to_{\rm fdd} V_{\alpha,\beta}\, \tau \hskip1cm  \text{as } N,n, N^{1/\beta}/n \to \infty,
\hskip1.75cm \text{for } 0 < \beta < 1,\label{S0} \\
(N n)^{-1/\alpha} S^+_{N,n} (\tau) &\to_{\rm fdd} \kappa^{1/\alpha}_\alpha \zeta_\alpha (\tau) \quad \text{as } N,n \to \infty
\text{ in arbitrary way,} \quad \text{for }  \beta > \alpha.\label{S1}
\end{align}

\medskip

\noi {\it Proof of \eqref{i1}} (case $\mu = \infty$) follows from \eqref{S0} and  \eqref{S1}
since  $\mu_{N,n} \to \infty $ implies       
$(Nn)^{1/\alpha} = o (A_{N,n})$.

\medskip

\noi {\it Proof of \eqref{ii1}} (case $\mu = 0$)
follows in view of \eqref{S1},
if we prove that 
\begin{equation}\label{rmd}
S_{N,n}^- (\tau) = o_{\rm p} ( A_{N,n} ).
\end{equation}
For any $\theta \in \R$, consider
\begin{align*}
\Theta_{N,n} &:= N \E \big[ \exp \big\{ \i \theta A_{N,n}^{-1} S_{1,n}^- (\tau) \big\} - 1 \big]\\
&= \psi_1 N \int_0^1 \big( \exp \big\{ - A_{N,n}^{-\alpha} \sum_{s \le 0} | c_{[n\tau]} (u,s) |^\alpha |\theta|^\alpha \big\} - 1 \big) (1-u)^{\beta-1} \d u,
\end{align*}
where $c_{[n\tau]}(u,s)$ is given in \eqref{ineq:dct}.
Using \eqref{ineq:dct}--\eqref{ineq:exp} and changing a variable, we obtain
$$
|\Theta_{N,n}| \le C N \int_0^1 \min \big( 1, \frac{n^{\alpha-1}}{N(1-u)} \big) (1-u)^{\beta-1} \d u \le C \int_0^{N^{1/\beta}} \min \big( 1, \frac{\mu_{N,n}^{1-\alpha}}{x} \big) x^{\beta-1} \d x = o (1).
$$
This proves \eqref{rmd}, hence, \eqref{ii1}.	

\medskip
	
\noi {\it Proof of \eqref{iii1}} (case $\mu \in (0,\infty)$).
For $N,n$ large enough, decompose $\Theta_{N,n} = \Theta_{N,n}^- + \Theta_{N,n}^+$ in \eqref{ThetaNn}, with $\psi_1 (1-u)^{\beta -1}$
replaced by $\phi(u)$, as
\begin{align*}
\Theta_{N,n}^- := N \int_0^1 ( \e^{ - K_{N,n}^- (u) } - 1 ) \phi (u) \d u, \quad \Theta_{N,n}^+ := N \int_0^1 \e^{ - K_{N,n}^- (u) } ( \e^{ - K_{N,n}^+ (u) } - 1 ) \phi (u) \d u,
\end{align*}
where
$$
K_{N,n}^- (u) := A_{N,n}^{-\alpha} \sum_{s\le 0} |\vartheta_n(u,s)|^\alpha, 
\quad K_{N,n}^+ (u) := A_{N,n}^{-\alpha} \sum_{s > 0} |\vartheta_n(u,s)|^\alpha.  
$$
Since $\mu_{N,n} \to \mu \in (0,\infty)$ implies $N^{1/\beta}/n \to \infty$, by \eqref{S0} we have that
$\Theta_{N,n}^- \to \log \E \e^{\i \mu^{(1/\alpha)- 1} (\sum_{j=1}^d \theta_j \tau_j) V_{\alpha,\beta}}.$
Next, using $K_{N,n}^- (u)  \to 0$, $|\e^{-K_{N,n}^- (u)} | \le C$, $u \in [0,1)$, similarly to the proof of \eqref{S1},
we obtain $\Theta_{N,n}^+ \to \log \E \e^{\i \kappa^{1/\alpha}_\alpha \sum_{j=1}^d \theta_j \zeta_\alpha (\tau_j) }$. Hence,
\eqref{iii1} follows, including the independence of $V_{\alpha,\beta}$ and $\{\zeta_\alpha(\tau),\, \tau \ge 0\}$.
Theorem~\ref{thm2} is proved.
\end{proof}


\begin{proof}[Proof of Proposition~\ref{prop1}] As noted in Section~\ref{sec3}, for $\alpha =2 $, the proposition is
proved in
\cite[Propositions~3.1, 3.2]{pili14}. The subsequent proof for $0 < \alpha < 2$ uses similar argument.
	
	\medskip
	
\noi (i) Write ${\cal Z}_{\alpha,\beta} (\tau) = {\cal Z}^-_{\alpha,\beta} (\tau) + {\cal Z}^+_{\alpha,\beta}(\tau), $
where ${\cal Z}^+_{\alpha,\beta} (\tau) :=
\int_{[1,\infty) \times D(\R)} z (\tau; x ) (N(\d x, \d \zeta)- \nu (\d x, \d \zeta) \1 (\alpha > 1) ), $  $
{\cal Z}^-_{\alpha,\beta} (\tau)
:= \int_{(0,1) \times D(\R)} z (\tau; x ) N(\d x, \d \zeta)$. Next, let $I(p,\tau) := I^-(p,\tau) + I^+(p,\tau), $ where
$
I^+(p,\tau) := \int_1^\infty \E_\alpha |z(\tau;x)|^p x^{\beta-1} \d x$, $I^-(p,\tau) := \int_0^1 \E_\alpha |z(\tau;x)|^p x^{\beta-1} \d x$.  Then ${\cal Z}^\pm_{\alpha,\beta} (\tau)$ are well defined
if $I^\pm (p,\tau) < \infty $ for some $0< p < \alpha$ in which case
\begin{eqnarray} \label{EZp}
&\E |{\cal Z}^\pm_{\alpha,\beta}(\tau)|^p \le
C I^\pm (p, \tau)
\end{eqnarray}
with $C = C(p)$ depending on $p$ alone;
see \cite[(3.3)]{pili14}. By well-known property of
an $\alpha$-stable stochastic integral in \eqref{zdef2},
$ \E_\alpha | z(\tau;x) |^p = C (\int_\R |f_\tau (x,s)|^\alpha \d s)^{p/\alpha}$, $\forall p \in(0, \alpha)$; see \cite[Property~1.2.17, Proposition~3.4.1]{samo1994}.
Then, using \eqref{ineq:exp}, we obtain
\begin{align}\label{ineq:f}
\int_{\R} |f_\tau (x,s)|^\alpha \d s&= \frac{(1-\e^{-x\tau})^\alpha}{\alpha x^{1+\alpha}} + \frac{1}{x^{\alpha}} \int_0^\tau ( 1-\e^{-xs} )^\alpha \d s \le \frac{1}{\alpha} \min \big ( \frac{\tau^\alpha}{x}, \frac{1}{x^{1+\alpha}} \big) + \min \big(  \frac{\tau}{x^\alpha}, \tau^{1+\alpha} \big).
\end{align}
Let first $1< \alpha < 2$. Then  \eqref{ineq:f} simplifies to
$\int_\R |f_\tau (x,s)|^\alpha \d s \le C \min ( \frac{\tau^\alpha}{x}, \frac{\tau}{x^\alpha} ) $
implying
\begin{eqnarray}\label{Ip}
&I(p,\tau)\le C \big(\tau^p \int_0^{1/\tau}  x^{\beta-1-p/\alpha} \d x + \tau^{p/\alpha} \int_{1/\tau}^\infty x^{\beta-1 -p} \d x \big)
\le C \tau^{p + p/\alpha - \beta}
\end{eqnarray}
for $\beta < p < \min (\alpha, \alpha \beta) = \alpha \min(1, \beta),$
with $C >0$ independent of $\tau >0$. Obviously, for given $0< \beta < \alpha $, such $p $ exists implying the existence of the Poisson stochastic integrals $
{\cal Z}^\pm_{\alpha,\beta} (\tau)$ and ${\cal Z}_{\alpha,\beta} (\tau)$.

Next, let $0< \alpha \le 1 $. Here, we need to discuss the existence of  ${\cal Z}^\pm_{\alpha,\beta} (\tau)$ separately. From
\eqref{ineq:f}, we have
\begin{eqnarray}\label{Ip1}
&I^+(p^+,1) \le  C \int_1^\infty  x^{\beta-1-p^+} \d x \le \infty,  \quad \beta < p^+ < \alpha, \\
&\ I^-(p^-,1) \le C \int_0^1  x^{\beta-1-p^-/\alpha} \d x
\le \infty,  \quad p^- < \alpha\beta < \alpha. \nn
 \end{eqnarray}
Clearly, for any $0< \beta < \alpha \le 1, $ such $p^\pm $ satisfying \eqref{Ip1} exist, implying the existence of
${\cal Z}^\pm_{\alpha,\beta} (\tau)$ and ${\cal Z}_{\alpha,\beta} (\tau)$ for $\tau=1$, and the last result extends
to all $\tau >0$ in an obvious way.   The stationarity of increments is immediate from \eqref{calZfd}.
Infinite divisibility and the form of the characteristic function in \eqref{calZfd}	follow from general properties
of Poisson stochastic integrals, see e.g.\ \cite[(3.1)]{pili14}.

\medskip
	
\noi (ii) Follows from \eqref{Ip}, \eqref{Ip1}  and \eqref{EZp}.

\medskip

\noi (iii) Follows from (ii) and $\E |{\cal Z}_{\alpha,\beta}(\tau)| \le CI(1,\tau) < \infty $ since
$1< \alpha \min (\beta, 1)$ is equivalent to $\alpha > 1, \alpha \beta > 1 $.

\medskip
	
\noi (iv) Follows from Kolmogorov's criterion,  stationarity of increments of ${\cal Z}_{\alpha,\beta}$, and \eqref{EZp}, \eqref{Ip} since, for
$1/\alpha < \beta < \alpha$,  we can find $p$ sufficiently close to $\alpha \min(\beta, 1)$ such that the exponent of $\tau$
on the r.h.s.\ of \eqref{Ip} is greater than 1: $p+p/\alpha -\beta > 1$.

\medskip
	
\noi (v) For brevity,
we assume $\omega (\theta) \equiv 1$,
$\psi_1 = 1$ and restrict the proof to one-dimensional convergence at $\tau > 0$.
Let $1 < \beta < \alpha$ and $H := 1-\frac{\beta-1}{\alpha} > 0$. For any $\theta \in \R$, we have
$$
\E \e^{ \i \theta c^{-H} {\cal Z}_{\alpha,\beta} (c \tau) }
= \exp \Big\{ \int_{\R_+} \big( \exp \big\{ - c^{\beta} | \theta|^\alpha \int_{\R} |f_{\tau} (c x,s)|^\alpha \d s  \big\} -1 \big) x^{\beta-1} \d x \Big\}
$$
since $\int_\R |f_{c \tau }(x,s)|^\alpha \d s = c^{1+\alpha} \int_\R |f_{\tau }(c x,s)|^\alpha \d s$.
By change of variables, we further rewrite
\begin{align*}
\E \e^{ \i \theta c^{-H} {\cal Z}_{\alpha,\beta}(c \tau) }
&= \exp \Big\{ \int_{\R_+} c^{-\beta} \big( \exp \big\{ - c^{\beta} | \theta|^\alpha \int_{\R} |f_{\tau} (x,s)|^\alpha \d s  \big\} -1 \big) x^{\beta-1} \d x \Big\}\\
&\to \exp \Big\{ - | \theta|^\alpha \int_{\R_+ \times \R} |f_{\tau} (x,s)|^\alpha x^{\beta-1} \d x \d s \Big\} =
\E \e^{\i \theta \Lambda_{\alpha,\beta} (\tau)}, \quad c \to 0,
\end{align*}
where the convergence follows by the DCT (the domination can be verified using \eqref{ineq:exp} and \eqref{ineq:f}).
	
Next, let $0 < \beta < \min(\alpha, 1)$. For any $\theta \in \R$, we have
$
\E \e^{\i \theta c^{-1} {\cal Z}_{\alpha,\beta}(c \tau)} = \exp \{ \int_{\R_+} ( \e^{ - |\theta|^\alpha K_c (x) } - 1 ) x^{\beta-1} \d x \},
$
where
$$
K_c (x) := \frac{1}{c^\alpha} \int_{\R} |f_{c\tau} (x,s)|^\alpha \d s = \frac{1}{\alpha x} \big( \frac{1-\e^{-c x \tau}}{cx} \big)^\alpha + \int_0^{c \tau} \big( \frac{1 - \e^{-xs}}{cx} \big)^\alpha \d s
\to \frac{\tau^\alpha}{\alpha x}, \quad c \to 0.
$$
For $K_c(x)$, we can find a dominating function using \eqref{ineq:exp} and \eqref{ineq:f}, because the latter inequality gives
$K_c (x) \le C x^{-1}$ if $\alpha \ge 1$ and $K_c (x) \le C \max (x^{-1}, c^{1-\alpha} x^{-\alpha}) \le C x^{-\alpha}$ if $\alpha < 1$ for $x > 1$.
Then, by the DCT, we obtain $\E \e^{\i c^{-1} {\cal Z}_{\alpha,\beta}(c \tau)} \to \exp \{ \int_{\R_+} (\e^{-(\alpha x)^{-1} \tau^\alpha |\theta|^\alpha} - 1) x^{\beta-1} \d x \} = \E \e^{\i \theta \tau V_{\alpha,\beta}}$, $c \to 0$, see \eqref{Wdef}--\eqref{VWchf}.


Let $1 = \beta < \alpha$. For any $\theta \in \R$, we consider $\E \e^{\i \theta (\log(1/c) )^{-1/\alpha} c^{-1}{\cal Z}_{\alpha,\beta}(c \tau)} = \e^{\sum_{i=1}^3 I_{c,\delta}^{i}}$,
	where
	\begin{align*}
	\sum_{i=1}^3I_{c,\delta}^{i} :=  \big\{ \int_0^\delta + \int_{\frac{\delta}{c}}^\infty  + \int_\delta^{\frac{\delta}{c}} \big\} \big( \e^{-|\theta|^\alpha (\log\frac{1}{c})^{-1} K_c(x) } - 1 \big) \d x \quad \text{with } K_c (x) := \frac{1}{c^\alpha} \int_{\R} |f_{c\tau} (x,s)|^\alpha \d s,
	\end{align*}
	the same as above, given a $\delta > 0$. Then $I_{c,\delta}^1 = o(1)$ and by \eqref{ineq:exp}, \eqref{ineq:f},
	\begin{align*}
	|I^2_{c,\delta}| &\le \frac{C}{c^\alpha \log \frac{1}{c}} \int_{\frac{\delta}{c}}^\infty \big( \frac{1}{\alpha x} \min \big( c \tau, \frac{1}{x} \big)^{\alpha} + c \tau \min \big( \frac{1}{x}, c\tau \big)^\alpha \big) \d x\\
	&\le \frac{C}{\log \frac{1}{c}} \int_{\delta}^\infty \big( \frac{1}{\alpha x} \min \big( \tau, \frac{1}{x} \big)^{\alpha} + \tau \min \big( \frac{1}{x}, \tau \big)^\alpha \big) \d x = o(1), \quad c \to 0.
	\end{align*}
	With the notation $\tilde K (cx):= x K_c (x)$, $x \in \R_+$, using \eqref{ineq:exp} and $\tilde K (c x) \le C$, $x>\delta$, we have that
	$$
	I_{c,\delta}^3 \sim -|\theta|^\alpha \int_\delta^{\frac{\delta}{c}} \frac{\tilde K(cx)}{x \log \frac{1}{c}} \d x, \quad c \to 0,
	$$
	for every $\delta >0$. Furthermore, $\limsup_{c \to 0} \sup_{\delta <x<\frac{\delta}{c}} |\tilde K(cx) - \frac{\tau^\alpha}{\alpha} | < \epsilon(\delta)$ with $\lim_{\delta \to 0} \epsilon(\delta) = 0$, since
	$$
	\tilde K (w) = \frac{1}{\alpha} \big( \frac{1-\e^{-w \tau}}{w} \big)^\alpha + w \int_0^{\tau} \big( \frac{1 - \e^{-w s} }{w} \big)^\alpha \d s
	\to \frac{\tau^\alpha}{\alpha}, \quad w \to 0.
	$$
	Hence, $\lim_{\delta \to 0} \limsup_{c \to 0} |I_{c,\delta}^3 + \frac{\tau^\alpha|\theta|^\alpha}{\alpha}| = 0$, finishing the proof of
	$
	\lim_{c \to 0} \E \e^{\i \theta (\log (1/c))^{-1/\alpha} c^{-1} {\cal Z}_{\alpha,\beta}(c \tau)} = \E \e^{\i \theta \tau V_{\alpha,1}}.
	$
	
Finally, consider the large scale limit of  ${\cal  Z}_{\alpha,\beta}$ as $c \to \infty $ for $0 < \beta < \alpha$. Then $\E \e^{\i \theta c^{-1/\alpha} {\cal Z}_{\alpha,\beta} (c \tau)} = \exp \{ \int_{\R_+} (\e^{-|\theta|^\alpha K_c (x) } - 1) x^{\beta-1} \d x \}$, where, using the scaling property,
	\begin{align*}
	K_c (x) &:= \frac{1}{c} \int_{\R} |f_{c\tau} (x,s) |^\alpha \d s = c^{\alpha} \int_{\R} |f_{\tau} (c x,s) |^\alpha \d s =
	\frac{(1-\e^{- c x \tau})^\alpha}{c \alpha x^{1+\alpha}} + \frac{1}{x^\alpha} \int_0^{\tau} ( 1 - \e^{-c xs} )^\alpha \d s \to \frac{\tau}{x^\alpha}, \quad c \to \infty.
	\end{align*}
It is obvious that, for $x>1$, $K_c(x)\leq x^{-\alpha}((c\tau)^{-1}+\tau)=O(x^{-\alpha})$. Therefore, by the DCT, $\E \e^{\i \theta c^{-1/\alpha} {\cal Z}(c \tau)} \to \exp \{ \int_{\R_+} ( \e^{- |\theta|^\alpha x^{-\alpha} \tau} - 1 ) x^{\beta-1} \d x \} = \E \e^{\i \theta {\cal W}_{\alpha,\beta} (\tau)}$, $c \to \infty$, see \eqref{Wdef}--\eqref{VWchf}. This proves part (v) and completes the proof of Proposition~\ref{prop1}.
\end{proof}

\bigskip

\section*{Acknowledgments}

We thank an anonymous referee for useful comments.
Vytaut{\.e} Pilipauskait{\.e} acknowledges the financial support from the project ``Ambit fields: probabilistic properties and statistical inference'' funded by Villum Fonden.



\bigskip
\footnotesize


\begin{thebibliography}{9}	
	
\bibitem{barn01} O.E. Barndorff-Nielsen. Superposition of Ornstein-Uhlenbeck type processes. {\em Theory Probab. Appl.} {\bf 45} 
(2001), 175--194.

\bibitem{ber2010} J. Beran, M. Sch\"utzner and S. Ghosh.
From short to long memory: Aggregation and estimation.
{\em Comput. Stat. Data Anal.}
{\bf  54} (2010), 2432--2442.

\bibitem{ber13} J. Beran, Y. Feng, S. Gosh and R. Kulik.
{\em Long-Memory Processes: Probabilistic Properties and Statistical Methods.} Springer, New York 2013.
	
\bibitem{celo07} D. Celov, R. Leipus, and A. Philippe. Time series aggregation, disaggregation, and long memory. {\em Lith. Math. J.} {\bf 47} (2007), 379--393.

\bibitem {celo2010} D. Celov, R. Leipus,and A. Philippe.
Asymptotic normality of the mixture density estimator in a
disaggregation scheme. {\em J. Nonparametric Statist.} {\bf 22} (2010), 425--442.

\bibitem{domb11} C. Dombry and I. Kaj. The on-off network traffic under intermediate scaling. {\em Queueing Sys.} {\bf 69} (2011), 29--44.

\bibitem{emb80} P. Embrechts and C.M. Goldie. On closure and factorization properties of subexponential and related 
distributions. {\em J. Austral. Math. Soc. (Series A)} {\bf 29} (1980), 243--256. 


\bibitem{fell71} W. Feller. {\em An Introduction to Probability Theory and Its Applications}, volume 2, 2nd ed. Wiley, New York, 1971.
	
\bibitem{gaig06} R. Gaigalas. A Poisson bridge between fractional Brownian motion and stable L\'evy motion. {\em Stochastic Process. Appl. } {\bf 116} (2006), 447--462.
	
\bibitem{gaig03} R. Gaigalas and I. Kaj. Convergence of scaled renewal processes and a packet arrival model. {\em Bernoulli} {\bf 9} (2003), 671--703.
	
	
	
\bibitem{grah18} D. Grahovac, N.N. Leonenko and M.S. Taqqu. The multifaceted behavior of integrated supOU processes: The infinite variance case. Preprint (2018+). {\tt arXiv:1711.09623v1 [math.PR]}
	
\bibitem{gran80} C.W.J. Granger. Long memory relationship and the aggregation of dynamic models. {\em J. Econometrics} {\bf 14} (1980), 227--238.
	
\bibitem{ibra71} I. Ibragimov and Y. Linnik. \textit{Independent and Stationary Sequences of Random Variables.} Wolters-Noordhoff, Groningen, 1971.
	
\bibitem{jir13} M. Jirak. Limit theorems for aggregated linear processes.
{\em Adv. Appl. Probab.} {\bf 45} (2013), 520--544.
	
\bibitem{kaj05} I. Kaj. Limiting fractal random processes in heavy-tailed systems. In J. L{\'e}vy-V{\'e}hel and E. Lutton, editors, {\em Fractals in Engineering: New Trends in Theory and Applications}, pages 199--217. Springer London, 2005.
	
\bibitem{kaj08} I. Kaj and M.S. Taqqu. Convergence to fractional Brownian motion and to the Telecom process: the integral representation approach. In V. Sidoravicius and M. E. Vares, editors, {\it In and Out of Equilibrium 2}, pages 383--427. Birkh\"auser Basel, 2008.
	
	
	
	
\bibitem{levy00} J.B. Levy and M.S. Taqqu. Renewal reward processes with heavy-tailed interrenewal times and heavy-tailed rewards. {\em Bernoulli} {\bf 6} (2000), 23--44.

\bibitem{lei2006} R. Leipus, G. Oppenheim, A. Philippe and  M.-C. Viano.  Orthogonal series density estimation in a disaggregation scheme.
{\em J. Statist. Plan. Inf.} {\bf 136}  (2006), 2547--2571.

\bibitem{lei2014} R. Leipus, A. Philippe, D. Puplinskait\.e  and D. Surgailis.
Aggregation and long memory: recent developments. {\em J. Indian Statistical Association} {\bf 52} (2014), 71--101.

\bibitem{lei2016} R. Leipus, A. Philippe, V. Pilipauskait\.e  and D. Surgailis.
Nonparametric estimation of the distribution of the autoregressive coefficient from panel random-coefficient AR(1) data. {\em J. Multiv. Anal.} {\bf 153}  (2017), 121--135.

\bibitem{lei2018a} R. Leipus, A. Philippe,  V. Pilipauskait\.e and  D. Surgailis.
Estimating long memory in panel random-coefficient AR(1) data. Preprint (2018+). {\tt arXiv:1710.09735v2 [math.ST]}

\bibitem{lei2018b}  R. Leipus, A. Philippe, V. Pilipauskait\.e and  D. Surgailis.
Sample autocovariances of random-coefficient AR(1) panel model. To appear in {\it Electron. J. Stat.} (2019+). {\tt arXiv:1810.11204v1 [math.ST]}

	
\bibitem{miko02} T. Mikosch, S. Resnick, H. Rootz\'en and A. Stegeman. Is network traffic approximated by stable L\'evy motion or fractional Brownian motion? {\em Ann. Appl. Probab.} {\bf 12} (2002), 23--68.
	
\bibitem{miko03} T. Mikosch. Modelling dependence and tails of financial time series. In B. Finkenst{\"{a}}dt and H. Rootz{\'{e}}n, editors, {\em Extreme Values in Finance, Telecommunications and the Environment}, pages 185--286. Chapman \& Hall, New York,  2003.

\bibitem{oppe04} G. Oppenheim and M.-C. Viano. Aggregation of random parameters Ornstein-Uhlenbeck or AR processes: some convergence results. {\em J. Time Ser. Anal.} {\bf 25} (2004), 335--350.

\bibitem{pili17} V. Pilipauskait\.e. {\em Limit Theorems for Spatio-Temporal Models with Long-Range
Dependence}. Doctoral dissertation, Vilnius 2017.

\bibitem{pili14} V. Pilipauskait\.e and D. Surgailis. Joint temporal and contemporaneous aggregation of random-coefficient AR(1) processes. {\em Stochastic Process. Appl.} {\bf 124} (2014), 1011--1035.


\bibitem{pili16} V. Pilipauskait\.e and D. Surgailis. Anisotropic scaling of random grain model  with application
to network traffic. {\em J. Appl. Probab.} {\bf 53} (2016), 857--879.

\bibitem{pipi04} V. Pipiras, M.S. Taqqu, and L.B. Levy. Slow, fast and arbitrary growth conditions for renewal reward processes when the renewals and the rewards are heavy-tailed. {\em Bernoulli} {\bf 10} (2004), 121--163.
	
\bibitem{pupl09} D. Puplinskait\.e and D. Surgailis. Aggregation of random coefficient  AR1(1) process with infinite variance and common innovations. {\em Lith. Math. J.} {\bf 49} (2009), 446--463.
	
\bibitem{pupl10} D. Puplinskait\.e and D. Surgailis. Aggregation of random coefficient  AR1(1) process with infinite variance and idiosyncratic innovations. {\em Adv. Appl. Probab.} {\bf 42} (2010), 509--527.
	
	
\bibitem{rajp89} B.S. Rajput and J. Rosinski. Spectral representations of infinitely divisible processes. {\em Probab. Theory Related Fields} {\bf 82} (1989), 451--487.
	
\bibitem{robi78} P.M. Robinson.  Statistical inference for a random coefficient autoregressive model. {\em Scand. J. Statist.} {\bf 5} (1978), 163--168.
	
	
\bibitem{samo1994} G. Samorodnitsky and M.S. Taqqu. {\em Stable Non-Gaussian Random Processes.} Chapman \& Hall, New York, 1994.
	
\bibitem{taqq97} M.S. Taqqu, W. Willinger and R. Sherman. Proof of a fundamental result in self-similar traffic modeling. {\it Comput. Commun. Rev.} {\bf 27} 
(1997) 5--23.
	
	
\bibitem{zaff04} P. Zaffaroni. Contemporaneous aggregation of linear dynamic models in large economies. {\em J. Econometrics} {\bf 120} (2004), 75--102.
	
	
\bibitem{zolo86} V.M. Zolotarev. {\em One-Dimensional Stable Distributions}, Amer. Math. Soc., Providence, RI, 1986.
\end{thebibliography}
\end{document}